\documentclass[12pt]{report}

\usepackage{amsmath,amsthm,amssymb,amscd}
\usepackage{xy, xypic}
\usepackage{indentfirst}
\usepackage{enumerate}
\usepackage{color}
\usepackage{hyperref}
\usepackage{titlesec}

\newtheorem{theorem}{Theorem}[section]
\newtheorem*{theorem*}{Theorem}
\newtheorem{corollary}[theorem]{Corollary}
\newtheorem*{corollary*}{Corollary}
\newtheorem{lemma}[theorem]{Lemma}
\newtheorem{proposition}[theorem]{Proposition}

\newtheorem{conjecture}[theorem]{Conjecture}
\newtheorem*{conjecture*}{Conjecture}
\newtheorem{thm-dfn}[theorem]{Theorem-Definition}
\newtheorem{claim}[theorem]{Claim}

\theoremstyle{definition}
\newtheorem{definition}[theorem]{Definition}
\newtheorem{remark}[theorem]{Remark}
\newtheorem{exmple}[theorem]{Example}
\newtheorem{question}[theorem]{Question}

\numberwithin{equation}{section}

\newcommand{\quash}[1]{}  

\def\presuper#1#2%
  {\mathop{}%
   \mathopen{\vphantom{#2}}^{#1}%
   \kern-\scriptspace%
   #2}

\usepackage{afterpage}

\quash{
}

\newcommand{\is}{\simeq}
\newcommand{\tld}[1]{{\widetilde{#1}}} 

\newcommand{\frakb}{{\mathfrak b}}

\newcommand{\frakg}{{\mathfrak g}}
\newcommand{\frakh}{{\mathfrak h}}

\newcommand{\frakk}{{\mathfrak k}}

\newcommand{\frakm}{{\mathfrak m}}
\newcommand{\frakn}{{\mathfrak n}}

\newcommand{\frakt}{{\mathfrak t}}

\newcommand{\bbB}{{\mathbb B}}
\newcommand{\bbC}{{\mathbb C}}

\newcommand{\bbR}{{\mathbb R}}

\newcommand{\bbZ}{{\mathbb Z}}

\newcommand{\calA}{{\mathcal A}}
\newcommand{\calB}{{\mathcal B}}
\newcommand{\calC}{{\mathcal C}}
\newcommand{\calD}{{\mathcal D}}
\newcommand{\calE}{{\mathcal E}}
\newcommand{\calF}{{\mathcal F}}

\newcommand{\calL}{{\mathcal L}}
\newcommand{\calM}{{\mathcal M}}

\newcommand{\calO}{{\mathcal O}}
\newcommand{\calP}{{\mathcal P}}
\newcommand{\calQ}{{\mathcal Q}}

\newcommand{\calT}{{\mathcal T}}

\newcommand{\calZ}{{\mathcal Z}}

\newcommand{\wt}{{\frakh^*}} 

\begin{document}

\begin{titlepage}
	
	\centering
	
	{\scshape\LARGE On properties of the Casselman-Jacquet functor \par}
	
	\vspace{1cm}
	
	{\large Alexander Yom Din\par}
	
	\vspace{1cm}
	
	{School of Mathematical Sciences\par}
	{Raymond and Beverly Sackler Faculty of Exact Sciences\par}
	{Tel Aviv University\par}
	
	\vspace{2cm}
	
	\vfill
	
	Dissertation submitted to the Tel Aviv University Senate\par for the degree of Doctor of Philosophy\par
	
	\vspace{1.5cm}

	Prepared under the supervision of Prof. Joseph Bernstein

	\vfill
	
	{\large May 2016\par}
\end{titlepage}

\quash{\title{On properties of the Casselman-Jacquet functor}
\author{Alexander Yom Din}
\date{}
\maketitle}

\newpage
{\color{white}.}
\newpage

\chapter*{Acknowledgements}

The research was partially supported by the ERC grant 291612 and by the ISF grant 533/14.

I would like to thank my PhD advisor Joseph Bernstein for explaining his point of view on different mathematical topics and teaching various mathematical topics, for suggesting me the problem of studying the Casselman-Jacquet functor algebraically, and for being my PhD advisor.

I would like to thank Tsao-Hsien Chen, Shamgar Gurevich, Lena Gal, Adam Gal, Jiuzu Hong, Evgeny Musicantov, Eitan Sayag for various mathematical discussions.

\afterpage{\null\newpage}
\chapter*{Abstract}

In this thesis, we study the Casselman-Jacquet functor. We discuss a new technical approach which makes the Casselman-Jacquet functor right adjoint to the Bernstein functor. We give an explanation, using $D$-modules, of the Bruhat filtration appearing on the module obtained by applying the Casselman-Jacquet functor to a principal series representation. We record some conjectures.

\afterpage{\null\newpage}
\setcounter{tocdepth}{1} \tableofcontents



\pagebreak
\chapter{Introduction}

In this thesis, we study the Casselman-Jacquet functor. We propose a new technical approach to defining this functor, and study some of its properties.

For concreteness let us, in this introduction, set: $G := GL_n(\mathbb{R})$, $K$ - the orthogonal matrices, $B$ - the upper triangular matrices, $N$ (resp. $\bar{N}$) - the upper triangular (resp. lower triangular) unipotent matrices, $T$ - the diagonal matrices (also considered as $B/N$), $M := K \cap B$. By $W$ we denote the Weyl group of $G$. By Gothic letters we denote the corresponding complexified Lie algebras.

Let us note that in the thesis itself, after the introduction, we will pass to a purely algebraic setup (see section \ref{sec_setting}).

\section{The $p$-adic case}

In $p$-adic representation theory, that is, when we replace the above real group with $G = GL_n (F)$, where $F$ is a local non-archimedian field (and all the subgroups are defined as above), we are interested in the category of smooth representations of $G$. One of the main tools to study it is by using parabolic induction functors. These induction functors have left adjoints, the Jacquet functors. So, we have the Jacquet functor from the category of smooth $G$-representations to the category of smooth $T$-representations. This functor is defined by taking $N$-coinvariants, and intuitively speaking measures "how much a representation has to do with representations parabolically induced from $T$". One of the main properties of the Jacquet functor is that it is exact.

\section{The Casselman-Jacquet functor}

In our case of interest, the case of real groups and their representations, the category of interest, from an algebraic point of view, is the category $ \calM (\frakg , K)$ of Harish-Chandra $(\frakg , K)$-modules\footnote{"Harish-Chandra" is a finiteness condition; It means that the module is finitely-generated as a $U(\frakg)$-module and locally finite as a $\calZ(\frakg)$-module.}.

The first guess at defining a "Jacquet functor" would be to consider the functor $$ \calM(\frakg, K) \to \calM(\frakt , M)$$ given by $$ V \mapsto V / \frakn V.$$ However, this functor has the big drawback of being non-exact.

Casselman corrects the naive functor above, by considering a completion; He defines: $$\hat{J}(V) := \varprojlim V / \frakn^m V,$$ the $\frakn$-adic completion. To avoid completion (the above completed object is not a Harish-Chandra module in some usual sense, for example), we may consider a dual functor $J^*$; $J^*(V)$ is the submodule of $V^*$, consisting of vectors killed by high power of $\frakn$: $$ J^*(V) = \{ w\in V^* \ | \ \frakn^m w = 0 \ \text{for some} \ m\ge 1\}.$$ The very surprising feature, that contrasts with the $p$-adic situation, is that $J^*(V)$ carries naturally not merely an $(\frakt , M)$-module structure, but in fact a $(\frakg , MN)$-module structure. Casselman shows that $J^*$ is an exact contravariant functor $$ \calM (\frakg , K) \to \calM(\frakg , MN)$$ (see \cite{C2}).

The exactness of $J^*$, which is equivalent to the exactness of $\hat{J}$, is shown by a variant of the Artin-Rees lemma; See section \ref{sec_Jex} for more details.

Let us mention here that one sees in the literature a third "normalization" of the Casselman-Jacuqet functor, let us call it $J$. $J(V)$ is defined as the subspace of $\frakn$-finite vectors in $\varprojlim V / \bar{\frakn}^k V$ (see, for example, \cite{ENV}).

\section{My approach to the Casselman-Jacquet functor}

I define the Casselman-Jacquet functor slightly differently than the above, and denote it by $$ \calC: \calM(\frakg , K) \to \calM(\frakg , MN).$$ Namely, for a Harish-Chandra $(\frakg , K)$-module $V$, write $V=\oplus V_{\alpha}$ for the $K$-type decomposition, and define $\calC (V)$ as the subspace of vectors in $\prod V_{\alpha}$ which are $M$-finite and killed by high power of $\frakn$, multiplied by some "normalizing" one-dimensional vector space (see \ref{chp_CJB}).

This functor is easily related to $J^*$ (see proposition \ref{prop_relJC}).

An important feature of $\calC$, which (at least as far as I know) is not clear when one considers the previous definitions of the Casselman-Jacquet functor, is that $\calC$ admits a left adjoint, the so-called Bernstein functor $$\calB : \calM(\frakg , MN) \to \calM(\frakg , K)$$ (see chapter \ref{chp_CJB}).

I conjecture (see conjecture \ref{conj_dual2}) that $\calC$ is canonically isomorphic to $J$ (perhaps up to tensoring by a power of a one-dimensional vector space; see conjecture \ref{conj_dual2}). This conjecture is equivalent to "duality" for the Casselman-Jacquet functor  (see conjecture \ref{conj_dual}). In the case of real groups (as in this introduction), it seems that this conjecture is understood to some extent using analytic tools (see the last paragraph in Section 12 of \cite{C2}). But an algebraic understanding (formally, that would apply in the purely algebraic setting of section \ref{sec_setting}) seems to not exist yet.

Another interseting conjecture regarding $\calC$ is conjecture \ref{conj_exact}. Namely, in my approach we can consider $\calC$ from the category of all $(\frakg , K)$-modules to the category of all $(\frakg , MN)$-modules (not necessarily Harish-Chandra). The conjecture is that $\calC$ is exact on the subcategory of modules which are finitely-generated over $U(\frakg)$.

\section{D-module description}

Under the Beilinson-Bernstein correspondence, one can try to translate notions involving Harish-Chandra modules above to notions involving equivariant $D$-modules on the flag variety. The functor $J$ admits a $D$-module description, via nearby cycles (see \cite{ENV}).

In my current approach, however, I realize geometrically, instead of $\calC$, its left adjoint $\calB$, the Bernstein functor. So, we consider the "geometric Bernstein functor" $$ \bbB: \calM (D_{\lambda} , MN) \to \calM(D_{\lambda} , K)$$ (here $D_{\lambda}$ is a sheaf of twisted differential operators on the complex flag variety, and $\calM(D_{\lambda}, MN)$ is the category of $MN$-equivariant coherent $D_{\lambda}$-modules). The functor $\bbB$ admits a right adjoint $$\bbC: \calM(D_{\lambda} , K) \to \calM(D_{\lambda} , MN).$$

The conjecture that $\calC$ is isomorphic to $J$ is more or less equivalent to conjecture \ref{cnj_gjf}, that $\bbC$ is isomorphic to the geometric Jacquet functor of \cite{ENV}.

\section{The main result - $\calC$ applied to principal series}

The main result of this thesis is concerned with the structure of the $(\frakg , MN)$-module $\calC(P)$ that one obtains by applying $\calC$ to the $(\frakg , K)$-module $P$ corresponding to a principal series representation (i.e., parabolic induction from $T$).

Casselman observes that $\calC(P)$ admits a "Bruhat" filtration\footnote{a filtration whose subquotients are indexed by Bruhat cells (i.e. orbits of $B$) in the real flag variety $G/B$, which in turn bijectively correspond to elements of the Weyl group $W$.}. This he does analytically, as the so-called automatic continuity allows to replace $J^*(P)$ by a space of continuous functionals on the actual smooth principal series representation. The principal series representation consists of smooth sections of a line bundle on $G/B$, and its Casselman-Jacquet module thus consists of distributional sections of a line bundle on $G/B$, killed by high power of $\frakn$ - one can then consider the filtration on this space of distributional sections by support on closures of Bruhat cells (see \cite[section 14]{C2}).

The question we ask is how to understand this Bruhat filtration purely algebraically. I provide an answer. Translating the situation to equivariant $D$-modules (via Beilinson-Bernstein correspondence), I want to ask, what is the property of the $D$-module corresponding to $\calC(P)$, that guarantees that it has a canonical Bruhat filtration. I then notice that an $MN$-equivariant $D$-module $\calF$ always has a derived Bruhat filtration, whose $w$-th subquotient is $(i_w)_* (i_w)^! \calF$ (Here $w \in W$ is an element in the Weyl group, and $i_w : \calO_w \to X$ is the corresponding $MN$-orbit in the complex flag variety). In the case when $(i_w)^! \calF$ are all acyclic, i.e. concentrated in cohomological degree $0$, this derived filtration becomes a usual filtration in the abelian category. We call objects $\calF$ with this property (that $(i_w)^! \calF$ is acyclic for every $w \in W$) costandard-filtered. Another characterization of such objects is that they admit a filtration (in the abelian category) by costandard objects - those are objects of the form $(i_w)_* \calE$ with irreducible $\calE$.

\begin{theorem*}[Theorem \ref{thm_main}]
The object in $\calM(D_{\lambda} , MN)$ which one obtains when applying the geometric Casselman-Jacquet functor $\bbC$ to an object in $\calM(D_{\lambda},K)$ which corresponds to a principal series Harish-Chandra module via the Beilinson-Bernstein correspondence, is costandard-filtered. In particular, it admits a canonical Bruhat filtration. For $w \in W$, the $w$-th subquotient of this Bruhat filtration is a suitable costandard object.
\end{theorem*}

\begin{corollary*}[Corollary \ref{cor_main}]
The object in $\calM(\frakg , MN)$ that one obtains when applying the Casselman-Jacquet functor $\calC$ to a principal series Harish-Chandra module in $\calM(\frakg,K)$, admits a canonical Bruhat filtration. For $w \in W$, the $w$-th subquotient of this Bruhat filtration is a suitable twisted Verma module.
\end{corollary*}

Twisted Verma modules are a version, in $|W|$-flavours, of Verma modules; They have the same composition series, but are "glued" from the irreducible constituents differently. For example, the twisted Verma module corresponding to $1 \in W$ is just the Verma module, while the twisted Verma module corresponding to $w_0 \in W$ (the longest Weyl element) is the dual Verma module. See section \ref{sec_tvm} and \cite{AL}.

A better statement would be one that identifies the subquotients more "canonically" and "functorially". See conjectures \ref{cnj_geomlem} and \ref{cnj_geomlem2}.

Let me, finally, note the similarity of the above corollary to the geometric lemma from the representation theory of $p$-adic groups (see \cite[Ch. III, 1.2]{B}).


\pagebreak
\chapter{Setting and Notations}

In this chapter we fix the setting and notations for the thesis.

As opposed to the introduction, we will work in a completely algebraic setup (see example \ref{exm_real} for the connection to the setting of the introduction). We fix a field $k$, algebraically closed of characteristic $0$. All vector spaces, tensor products, algebraic varieties etc. will be over $k$. All algebraic varieties will be assumed quasi-projective, so that to be on the safe side with things like cohomological dimension of categories of $D$-modules.

\section{Setting}\label{sec_setting}

We fix the following data:

\begin{itemize}
	\item $G$ - a connected reductive algebraic group.
	\item $\theta: G \to G$ - an involution of $G$.
	\item $K$ - an open subgroup of $G^{\theta} = \{ g \in G \ | \ \theta(g) = g\}$.
	\item $B$ - a $\theta$-split Borel subgroup of $G$. This means that $\theta(B)$ is opposite to $B$. We assume furthermore that $\theta(t) = t^{-1}$ for all $t\in B \cap \theta(B)$.
\end{itemize}

\begin{remark}
	In the language of real groups, the condition that $B$ is $\theta$-split means that we are in the quasi-split case, while the further condition that $\theta(t) = t^{-1}$ for all $t\in B \cap \theta(B)$ means that we are in the split case.
\end{remark}

\begin{exmple}\label{exm_real}
	One can take $G=GL_n$, $\theta (g) = (g^t)^{-1}$, $K=G^{\theta}$ - orthogonal matrices, $B$ - upper triangular matrices.
	If $k = \bbC$, then we can think of $G$ as a real algebraic group. $K(\bbR)$ is then a maximal compact subgroup in $G(\bbR)$, and $B$ is a Borel subgroup defined over $\bbR$. We thus obtain the example used in the introduction (just that now, we denote by $K$ and etc. the complex algebraic groups, not the real points as in the introduction).
\end{exmple}

\section{Notations}

\subsection{General notations}

Denote by $N$ the unipotent radical of $B$, and $T := B \cap \theta(B)$ ($T$ is thus a $\theta$-split torus - meaning that $\theta(t)=t^{-1}$ for all $t\in T$).

Denote $M := K \cap B$. $M$ is a finite subgroup of $T$.

By Gothic letters, we denote the Lie algebras of the corresponding algebraic groups: $\frakg = Lie(G)$ etc.

The Weyl group of $G$ will be denoted by $W$. We identify $W \is N_G(T) / T$.

"The Cartan" Lie algebra of $\frakg$ will be denoted by $\frakh$. It is defined by choosing a Borel subgroup $B_0 \subset G$, and defining $\frakh$ as the Lie algebra of $B_0 / R_u (B_0)$. When one chooses a different Borel subgroup $B_1 \subset G$, an element $g \in G$ satisfying $g B_0 g^{-1} = B_1$ identifies between $B_0 / R_u (B_0)$ and $B_1 / R_u (B_1)$, and this identification does not depend on the choice of $g$. In this way, we get a well-defined abelian Lie algebra $\frakh$ - "the Cartan".

We can identify $\frakh \is \frakt$ using our Borel $B$, when it is convenient. $W$ acts linearly on $\frakh$ - when identifying $W \is N_G(T)/T$ and $\frakh \is \frakt$, this is just the adjoint action. Making an affine translation of this action so that $-\rho$ is the fixed point, instead of $0$, we obtain the so-called "dot-action".

By $\calZ(\frakg)$ we denote the center of $U(\frakg)$. characters of $\calZ(\frakg)$ are called infinitesimal characters. We identify as usual $W \backslash \backslash \frakh^{*} \is Specm(\calZ(\frakg))$, where for the quotient we consider the dot-action. For a weight $\lambda \in \wt$, we denote by $[\lambda]$ the corresponding infinitesimal character.

\quash{and for $\theta \in Specm(\calZ(\frakg))$ denote by $[\theta] \subset \frakh^*$ the corresponding $W$-orbit.}

\subsection{Categories of modules}

For a closed subgroup $H \subset G$, we denote by $Mod(\frakg , H)$ the abelian category of (strong)\footnote{In this thesis, we always consider strongly equivariant modules - i.e. the two possible actions of the Lie algebra of the acting group should coincide.} $(\frakg , H)$-modules.

By $\calM(\frakg , H) \subset Mod(\frakg , H)$ we denote the full subcategory of Harish-Chandra modules, i.e. modules which are finitely generated as $U(\frakg)$-modules, and locally finite as $\calZ(\frakg)$-modules - these two properties together imply that $\calZ(\frakg)$ acts via a quotient of finite dimension.

Given an infinitesimal character $\theta$, we denote by $Mod^{\theta}(\frakg , H) \subset Mod(\frakg , H)$ and $\calM^{\theta}(\frakg , H) \subset \calM(\frakg , H)$ the full subcategories of objects on which $\calZ(\frakg)$ acts by $\theta$.

We denote by $\calM(\frakt , M)$ the category of Harish-Chandra $(\frakt , M)$-modules - those are just the finite-dimensional (strong) $(\frakt , M)$-modules. For $\lambda \in \wt$, we denote by $\calM^{\lambda}(\frakt , M) \subset \calM(\frakt , M)$ the full subcategory of modules on which $\frakt$ acts by $\lambda - \rho$. Under the forgetful functor, $\calM^{\lambda}(\frakt , M)$ is equivalent to the category of finite-dimensional algebraic $M$-modules.

\subsection{The flag variety}

The flag variety of $\frakg$ will be denoted by $X$. For $x \in X$, we denote by $\frakb_x$ the corresponding Borel subalgebra, and $\frakn_x := [\frakb_x , \frakb_x]$. We denote by $x_0 \in X$ the point corresponding to our Borel subgroup $B$.

The $MN$-orbits on $X$ are the same as the $N$-orbits on $X$, and are in bijection with $W$. Namely, to $w \in W$ corresponds the $MN$-orbit $\calO_w$ passing via $wx_0$. Here, by $wx_0$ we mean $\widetilde{w} x_0$, where $\widetilde{w} \in N_G (T)$ is any representative of $w$ - the end result does not depend on this choice of representative. We denote by $i_w : \calO_w \to X$ the embedding.

There is a finite number of $K$-orbits on $X$, and they are affine. There is a unique open $K$-orbit $U$ - that which passes through $x_0$. We denote by $j : U \to X$ the embedding.

For $\lambda \in \wt$, we have the TDO $D_{\lambda}$ on $X$. We use the shifted convention, so that $D_{\rho} \is D_X$, the sheaf of differential operators.

For integral weight $\mu \in \wt$, we have the $G$-equivariant line bundle $\calL_{\mu}$ on $X$.

\subsection{Categories of $D_{\lambda}$-modules}

For a closed subgroup $H \subset G$, we denote by $Mod(D_{\lambda} , H)$ the abelian category of (strongly)\footnote{In this thesis, we always consider strongly equivariant modules - i.e. the two possible actions of the Lie algebra of the acting group should coincide.} $H$-equivariant $D_{\lambda}$-modules. By $\calM(D_{\lambda} , H) \subset Mod(D_{\lambda} , H)$ we denote the full subcategory of those $D_{\lambda}$-modules which are coherent. If $H$ acts on $X$ with finitely many orbits, then coherent $H$-equivariant $D_{\lambda}$-modules are the same as holonomic $H$-equivariant $D_{\lambda}$-modules.

For $w \in W$, the category $\calM (i_w^{\cdot} D_{\lambda}, MN)$ of $MN$-equivariant $D_{\lambda}$-modules on the orbit $\calO_w$ is equivalent, via taking the fiber at the point $w x_0$, to $\calM^{\lambda}(\frakt , M)$. In the same fashion, the category $\calM(j^{\cdot} D_{\lambda} , K)$ is equivalent, via taking the fiber at the point $x_0$, to $\calM^{\lambda}(\frakt , M)$. We denote these equivalences by $$ [\cdot]^{\lambda}_w : \calM^{\lambda}(\frakt , M) \approx \calM (i_w^{\cdot} D_{\lambda}, MN) : Fib_{w x_0}$$ and $$ [\cdot]^{\lambda} : \calM^{\lambda}(\frakt , M) \approx \calM (j^{\cdot} D_{\lambda}, K) : Fib_{x_0}.$$

We have the localization-globalization adjunction $$ \Lambda : \calM^{[\lambda]} (\frakg , H) \rightleftarrows \calM (D_{\lambda} , H): \Gamma,$$ which is an equivalence if $\lambda$ is regular antidominant.

\subsection{$D$-algebras}

See appendix \ref{app_dalg} for the notions we use regarding $D$-algebras. Our reference for $D$-algebras, $TDO$'s, etc. is \cite{BB}.

For a $D$-algebra $\calA$ on a smooth algebraic variety $Y$, we denote by $Mod(\calA)$ the abelian category of $\calA$-modules. If an affine algebraic group $H$ acts on $Y$ and $\calA$ is $H$-equivariant, we denote by $Mod(\calA , H)$ the abelian category of (strongly)\footnote{In this thesis, we always consider strongly equivariant modules - i.e. the two possible actions of the Lie algebra of the acting group should coincide.} $H$-equivariant $\calA$-modules.

Of course, the previous notations $Mod(\frakg , H)$ and $Mod(D_{\lambda} , H)$ are special cases of this $Mod(\calA , H)$ - the first one when $Y = pt$ and $\calA = U(\frakg)$, and the second one when $Y = X$ and $\calA = D_{\lambda}$.

\subsection{$t$-Categories}

A $t$-category is a triangulated category $\calT$ equipped with a $t$-structure $(\calT^{\leq 0} , \calT^{ \ge 0})$. Let us recall that $\calT^{\leq 0} \cap \calT^{\ge 0}$ is an Abelian category, called the heart of $\calT$. The $t$-category $\calT$ is said to be bounded if for every object $A \in \calT$ there exist $i , j \in \bbZ$ such that $A \in \calT^{\leq i} \cap \calT^{\ge j}$. An object $A \in \calT$ is said to be acyclic if $A \in \calT^{\leq 0} \cap \calT^{\ge 0}$, i.e. $A$ lies in the heart of the $t$-structure.

\subsection{Some other notations}

Given a $K$-equivariant algebra $A$ (such as $U(\frakg)$), we denote by $\presuper{k}{a}$ the result of the action of the element $k \in K$ on the element $a \in A$.

Given a $U(\frakn)$-module $V$, we say that a vector $v \in V$ is $\frakn$-torsion, if $\frakn^m v = 0$ for some $m \ge 1$.

For a smooth algebraic variety $Y$, we denote by $\calO_Y$ the sheaf of regular functions on $Y$, and $O(Y) := \Gamma ( \calO_Y )$. Also, we denote by $D_Y$ the sheaf of differential operators on $Y$, and $D(Y) := \Gamma ( D_Y )$.


\pagebreak
\chapter{Summary}

Let us summarize the constructions and propositions of this thesis.

\begin{enumerate}[(1)]

\item We define a pair of adjoint functors $$ \calB : \calM (\frakg , MN) \rightleftarrows \calM (\frakg , K) : \calC .$$

The functor $\calC$ is directly related to Casselman's functor $J^*$ from the introduction - we call it the Casselman-Jacquet functor. The functor $\calB$ is, more or less, Bernstein's functor (a "projective" analog of Zuckerman's functor) - and so we call it the Bernstein functor (see chapter \ref{chp_CJB}).

\item For a weight $\lambda \in \wt$, we define a pair of adjoint functors $$ \bbB : \calM (D_{\lambda} , MN) \rightleftarrows \calM (D_{\lambda} , K) : \bbC .$$ Those are the "geometric" Bernstein and Casselman-Jacquet functors (see chapter \ref{chp_CJB} for $\bbB$ and chapter \ref{chp_bbC} for $\bbC$).

\item $\calB$ and $\bbB$ are related by the localization functor $\Lambda$, while $\calC$ and $\bbC$ are related by the globalization functor $\Gamma$; i.e. the following diagrams naturally $2$-commute:

$$\xymatrix{\calM(D_{\lambda}, MN) \ar[r]^{\bbB} & \calM(D_{\lambda}, K) \\ \calM^{[\lambda]}(\frakg , MN) \ar[r]_{\calB} \ar[u]^{\Lambda} & \calM^{[\lambda]}(\frakg , K) \ar[u]_{\Lambda}} \quad \quad \xymatrix{\calM(D_{\lambda}, MN) \ar[d]_{\Gamma} & \calM(D_{\lambda}, K) \ar[l]_{\bbC} \ar[d]^{\Gamma} \\ \calM^{[\lambda]}(\frakg , MN) & \calM^{[\lambda]}(\frakg , K) \ar[l]^{\calC}}$$ (see lemma \ref{lem_Bloc} and lemma \ref{prop_CasGamma}).

\item The functors $\calC$ and $\bbC$ are exact\footnote{Of course, a priori, $\bbB$ and $\calB$ are right exact as left adjoint functors and $\calC$ and $\bbC$ are left exact, as right adjoint functors.} (see theorem \ref{thm_Cex} and claim \ref{clm_bbC}).

\item The category $\calM (D_{\lambda} , MN)$ admits a formalism of (co)standard objects, (co)standard-filtered objects, etc. Also, $\calD (D_{\lambda} , MN)$, the derived category of $MN$-equivariant holonomic $D_{\lambda}$-modules (as in \cite{BL}), is naturally equivalent to $D^b(\calM(D_{\lambda},MN))$ (see appendix \ref{app_str} and section \ref{sec_MN}).

\item The left derived functor $$ L\bbB : D^b (\calM (D_{\lambda} , MN)) \to D^b(\calM (D_{\lambda} , K))$$ sends standard-filtered objects to acyclic objects (i.e. objects concentrated in cohomological degree $0$). As a consequence, the functor $$ \calM (D_{\lambda} , MN) \leftarrow \calM (D_{\lambda} , K) : \bbC$$ sends injective objects to costandard-filtered objects (see propositions \ref{prop_LBtilt} and \ref{prop_Ctilt}).

\item The main theorem: For a $K$-equivariant irreducible smooth $D_{\lambda}$-module $\calE$ on the open $K$-orbit $U$, the object $j_* \calE \in \calM (D_{\lambda}, K)$ is injective.  Hence, $\bbC (j_* (\calE))$ is costandard-filtered. As a result, $\bbC (j_* (\calE))$ admits a canonical $W$-filtration - a filtration whose subquotients are indexed by the Weyl group $W$. The subquotients are costandard objects (see theorem \ref{thm_main}).

\item Corollary of the main theorem: By passing to global sections in the previous paragraph, we obtain that for a principal series module $P \in \calM^{[\lambda]} (\frakg , K)$, the module $\calC ( P)$ admits a canonical $W$-filtration, whose subquotients are twisted Verma modules (see corollary \ref{cor_main}).

\end{enumerate}

Moreover, in chapter \ref{chp_conj}, we make some speculations:

\begin{enumerate}[$\langle$1$\rangle$]

\item In our setup, the functor $\calC$ is actually defined on categories of "big"\footnote{i.e., dropping the conditions of $U(\frakg)$-finite generation and $\calZ(\frakg)$-local finiteness.} modules $$ Mod (\frakg , MN) \leftarrow Mod (\frakg , K) : \calC.$$ We expect that $\calC$ preserves finite-generation over $U(\frakg)$, and conjecture that $\calC$ is exact on the subcategory of modules finitely-generated over $U(\frakg)$.

\item We conjecture that $\calC$ commutes, in a certain sense, with duality. In the case of real groups, this corresponds to Casselman's canonical pairing, and so presumably understood to some degree by analytical means.

\item We conjecture that $\calC$ is more or less isomorphic to the functor $J$, which is usually understood as the Jacquet-Casselman functor (as in \cite{ENV}, for example). This conjecture is equivalent to the previous one.

\item To establish property (6) above, we use the derived geometric Bernstein functor $$ \bbB^{der} : \calD (D_{\lambda} , MN) \to \calD (D_{\lambda} , K);$$ Here, the categories $\calD (D_{\lambda} , \cdot)$ are the derived categories of equivariant $D_{\lambda}$-modules, as in \cite{BL}. One might stress that this derived functor is the "correct one" (rather than $L \bbB$). We conjecture that $\bbB^{der}$ admits a $t$-exact right adjoint $\bbC^{der}$ and, furthermore, that $\bbC^{der}$ is more or less the geometric Jacquet functor of \cite{ENV}.

\item We expect a more precise formulation of the main theorem, where the subquotients are not only determined up to a non-canonical isomorphism, but rather identified canonically and functorially (perhaps after some small choice). In this formulation, the main theorem becomes very similar to the geometric lemma from the representation theory of $p$-adic groups (see \cite[Ch. III, 1.2]{B}).

\end{enumerate}

\pagebreak
\chapter{Averaging and relaxing functors}\label{sec_func}

In this chapter we discuss the "averaging" and "relaxing" functors, which will serve us when we define the Casselman-Jacquet and Bernstein functors. The idea is that we want a unified framework for defining the Bernstein functor on modules, as well as the Bernstein functor on $D$-modules.

For the functors between categories of modules over $D$-algebras that we use, one should consult appendix \ref{app_dalg}.

We fix the following. $Y$ is a smooth algebraic variety, equipped with an action of an affine algebraic group $K$. $\calA$ is a $K$-equivariant $D$-algebra on $Y$, and $M \subset K$ is a closed subgroup. We also assume that $K/ M$ is affine and $M$ is reductive.

\section{Motivation}

We would like to define a functor $Mod (\calA , M) \to Mod (\calA , K)$. For the motivation, naively replace equivariant $\calA$-modules by invariant functions and pretend that $Y$ is finite. If we have an $M$-invariant function $f$ on $Y$, we can construct a $K$-invariant function $Av(f)$ on $Y$ by the formula $$Av(f) (y)  = \sum_{k\in K/M} f(ky).$$ We would like to describe a diagrammatic interpretation.

For a map $\phi : Z \to Y$ we have the operation $\phi^{\bullet}$ which sends a function $f$ on $Y$ to the function on $Z$ defined by $\phi^{\bullet}(f)(z)=f(\phi(z))$. We also have the operation $\phi_{\bullet}$ which sends a function $f$ on $Z$ to the function on $Y$ defined by $\phi_{\bullet}(f)(y)=\sum_{\phi(z)=y} f(z)$. Finally, if $\phi$ is the quotient map by a free action of a group $M$, we have the operation $\phi_+$ which sends an $M$-invariant function $f$ on $Z$ to the function on $Y$ defined by $\phi_+ (f)(y)=f(\bar{y})$, where $\bar{y}$ is any preimage of $y$ under $\phi$.

We consider the diagram  $$ \xymatrix{ K \times Y \ar[r]^{\widetilde{a}} \ar[dd]_{p} & K \times Y \ar[d]^{q} & \\ & K /M \times Y \ar[d]^{\widetilde{p}} \\ Y & Y } $$

Here the map $p$ is the projection, $\widetilde{a} (k,y) := (k,ky)$, the map $q$ is the quotient map on the first factor and the identity on the second factor, and $\widetilde{p}$ is the projection. We then have $$Av(f) = \widetilde{p}_{\bullet} q_{+} \widetilde{a}_{\bullet} p^{\bullet} (f).$$ Notice that $\widetilde{a}$ is an isomorphism, so that $\widetilde{a}_{\bullet}$ is just "a change of coordinates".

\section{The functors $Av^M_K$ and $av^M_K$}

To define our functors, we consider again the diagram from the previous section:

$$ \xymatrix{ K \times Y \ar[r]^{\widetilde{a}} \ar[dd]_{p} & K \times Y \ar[d]^{q} & \\ & K /M \times Y \ar[d]^{\widetilde{p}} \\ Y & Y } $$

where $p$ is the projection, $\widetilde{a} (k,y) := (k,ky)$, the map $q$ is the quotient map on the first factor and the identity on the second factor, and $\widetilde{p}$ is the projection.

We fix the following equivariant structures on the spaces in the diagram: on the left $Y$ we consider the $M$-action. On the left $K \times Y$, we consider the $(K \times M)$-action, where $K$ acts by $k^{\prime}(k,y) := (k^{\prime}k,y)$ and $M$ acts by $m(k,y):=(km^{-1},my)$. On the right $Y$ we consider the $K$-action. On $K/M \times Y$ we consider the $K$-action given by $k^{\prime} (\overline{k},y):=(\overline{k^{\prime}k},k^{\prime}y)$. On the right $K\times Y$ we consider the $(K \times M)$-action, where $K$ acts by $k^{\prime}(k,y) := (k^{\prime}k,k^{\prime}y)$ and $M$ acts by $m(k,y):=(km^{-1},y)$. Notice that the $(K\times M)$-actions on the two copies of $K \times Y$ are interchanged via $\widetilde{a}$.

We also fix the following $D$-algebras on the spaces in the diagram: On both copies of $Y$, we consider the $D$-algebra $\calA$. On the rest of the spaces, we consider the $D$-algebras gotten from $\calA$ by pullback along the vertical maps. Of course, since all our maps are projection maps, the resulting $D$-algebras on the two copies of $K \times Y$ are $D_K \boxtimes \calA$, and the $D$-algebra on $K/M \times Y$ is $D_{K/M} \boxtimes \calA$. We note that all those $D$-algebras are suitably equivariant and also, importantly, that $\widetilde{a}$ canonically interchanges the $D$-algebras on both copies of $K \times Y$ - this is equivalent to the data of $K$-equivariancy of $\calA$.

Consider now the diagram of functors: $$ \xymatrix{ Mod(p^{\cdot}\calA , K \times M) \ar[r]^{\widetilde{a}_+} & Mod (p^{\cdot} \calA , K \times M) \ar[d]^{q_+} & \\ & Mod (\widetilde{p}^{\cdot} \calA, K)  \ar@/_/[d]_{\widetilde{p}_*}  \ar@/^/[d]^{\widetilde{p}_{\flat}} \\ Mod(\calA,M) \ar[uu]^{p^{\circ}} & Mod(\calA,K) } $$

Let us remark that here $p^{\circ}$ and $q_+$ are induction equivalences, and $\widetilde{a}_+$ is an equivalence of "transport of structure" ($\widetilde{a}$ is an isomorphism). The functor $\widetilde{p}_*$ kills the action of vector fields on $K/M$ by way of coinvariants, while $\widetilde{p}_{\flat}$ kills the action of vector fields on $K/M$ by way of invariants.

\begin{definition}\

\begin{enumerate}[(1)]
\item Denote by $$Av^M_K: Mod(\calA , M) \to Mod(\calA, K)$$ the functor $\widetilde{p}_* \circ q_+ \circ \widetilde{a}_+ \circ p^{\circ}$.

\item Denote by $$ Mod(\calA , K) \leftarrow Mod(\calA, M) : av^M_K$$ the functor $\widetilde{p}_{\flat} \circ q_+ \circ \widetilde{a}_+ \circ p^{\circ}$.
\end{enumerate}

\end{definition}

\begin{remark}
One might call $av^M_K$ the "Zuckerman functor", and $Av^M_K$ the "Bernstein functor". But we just call them "averaging functors", and reserve the name "Bernstein functor" for the functor $\calB$ to come later.
\end{remark}

\section{The functors $Rel^K_M$ and $rel^K_M$}\label{sec_rel}

Our functors $Av^M_K$ and $av^M_K$ admit adjoints as follows:

\begin{definition}\
\begin{enumerate}[(1)]
\item Denote by $$Rel^K_M : Mod(\calA , K) \to Mod(\calA , M) $$ the functor $p_+ \circ \widetilde{a}^{\circ} \circ q^{\circ} \circ \widetilde{p}^{\circ}$, which is left adjoint to $av^M_K$.
\item In case $Y=pt$, denote by $$Mod(\calA , M) \leftarrow Mod(\calA , K) : rel^K_M$$ the functor $p_+ \circ \widetilde{a}^{\circ} \circ q^{\circ} \circ \widetilde{p}^{\flat}$, which is right adjoint to $Av^M_K$.
\end{enumerate}
\end{definition}

\begin{lemma}
The functor $Rel^K_M$ is naturally isomorphic to the forgetful functor.
\end{lemma}

\begin{proof}

We can describe the inversre equivalence to $$p^{\circ} : Mod(\calA , M) \to Mod(p^{\cdot} \calA , K \times M),$$ instead of as $p_+$, as the composition $$ Mod(p^{\cdot} \calA , K \times M) \xrightarrow{for} Mod(p^{\cdot} \calA , M) \xrightarrow{s^{\circ}} Mod(\calA , M).$$ Here, $for$ is the forgetful functor from the action of $K \times M$ to the action of $M$ via the diagonal embedding $M \to K \times M$, and $s : Y \to K \times Y$ is the section map $s(y) = (1,y)$.

Thus, $Rel^K_M$ is expressed as a composition of pullbacks $$ s^{\circ} \circ for \circ \widetilde{a}^{\circ} \circ q^{\circ} \circ \widetilde{p}^{\circ},$$ which is clearly the desired forgetful functor.

\end{proof}

\subsection{Formulas when $Y=pt$}\label{sec_frml}

Let us describe our four functors explicitly, in the case when $Y=pt$. The proofs of these formulas are just careful chases through the diagrams and definitions.

\begin{enumerate}[(1)]

\item \textbf{Description of $av^M_K$}: For $V \in Mod (\calA, M)$, we denote by $Fun_M(K,V)$ the space of algebraic functions $\phi :K \to V$ satisfying $$\phi (km)=m^{-1}\phi(k) \quad \text{for } k \in K, m \in M.$$ On it, we consider the $\calA$-action given by $$(a\phi)(k)=\presuper{k^{-1}}{a} \cdot \phi(k) \quad \text{for } a \in \calA, k \in K$$ and the $K$-action given by $$(h\phi)(k)=\phi(h^{-1}k) \quad \text{for } h,k \in K.$$ Now, the two $\frakk$-actions, one coming from differentiating the $K$-action and one coming from the $\calA$-action - do not in general agree; Their difference defines a $\frakk$-action.

Taking invariants w.r.t. this $\frakk$-action yields an $(\calA , K)$-module, which is $av^M_K (V)$: $$ av^M_K(V) \is Fun_M (K,V)^{\frakk}.$$

\item \textbf{Description of $Av^M_K$}: We assume that $K$ is reductive. The description is similar to that of $av^M_K$; we just need to take coinvariants w.r.t. the above $\frakk$-action, instead of invariants, and add a "normalizing" factor: $$ Av^M_K (V) \is \wedge^{\text{top}}(\frakk / \frakm)^* \otimes Fun_M (K, V)_{\frakk}.$$ Here, $\wedge^{\text{top}}(\frakk / \frakm)^*$ is the top exterior power of $(\frakk / \frakm)^*$ - a one-dimensional vector space (arising from the passage from left to right $D(K/M)$-modules). Since $K$ and $M$ are reductive this factor does not come into play in the actions.

\item \textbf{Description of $Rel^K_M$}: The functor $Rel^K_M$ is, as we explained above, the forgetful functor.

\item \textbf{Description of $rel^K_M$}: Again we assume that $K$ is reductive. Given $V \in Mod(\calA , K)$, the module $rel^K_M (V)$ is given by $$rel^K_M (V) \is \wedge^{\text{top}}(\frakk / \frakm) \otimes Hom_K ( O(K) , V)^{M-\text{fin.}}. $$ Here again the factor $\wedge^{\text{top}}(\frakk / \frakm)$ is "inert", i.e. does not play a role in the description of the actions. The notation $(\cdot)^{M-\text{fin.}}$ stands for taking $M$-finite vectors. The $K$-action on $O(K)$ w.r.t. which we take the $Hom$ is the left regular action. Let us specify the $(\calA , M)$-action on $Hom_K ( O(K) , V)$. The $M$-action is the one induced by the right regular $M$-action on $O(K)$. To specify the $\calA$-action we prefer to describe $Hom_K ( O(K) , V)$ in a different way first. Namely, by Peter-Weyl there is a canonical identification $$Hom_K ( O(K) , V) \is \prod_{\alpha} V_{\alpha}$$ where $\alpha$ runs over isomorphism classes if irreducible algebraic $K$-modules, and $V_{\alpha}$ denotes the corresponding isotypic component  of $V$ (the so-called "$K$-type"). In this presentation, the action of $\calA$ is just "type-wise", i.e. the action of every $a \in \calA$ is such that the diagram $$ \xymatrix{ \prod V_{\alpha} \ar[r]^{a} & \prod V_{\alpha} \\ \oplus V_{\alpha} \ar[r]^{a} \ar[u] & \oplus V_{\alpha} \ar[u]}$$ commutes. In other words,  the action on the product extends the original action on the sum "by continuity" (thinking of the product as a completion of the sum). Notice that, in order for this to be well defined, it should hold that given $\alpha$ and $a \in \calA$, there is a finite number of $\beta$ such that $a \cdot V_{\beta} \cap V_{\alpha} \neq 0$. This is easy to deduce from $\calA$ being an algebraic $K$-module, by considering highest weights of the $K$-modules.

Thus we can summarize that $$rel^K_M (V) \is \wedge^{\text{top}}(\frakk / \frakm) \otimes \left( \prod_{\alpha} V_{\alpha} \right)^{M-\text{fin.}}. $$ where the action of $(\calA , M)$ is "type-wise".

\end{enumerate}

\begin{lemma}\label{lem_nilpind}
	Suppose that $N \subset K$ is a connected unipotent subgroup satisfying $K = M \ltimes N$. Then $av^M_K (V)$, via the counit morphism $Rel^K_M av^M_K (V) \to V$, is isomorphic to the subspace of $V$, consisting of $\frakn$-torsion vectors.
\end{lemma}

\begin{proof}
	This follows easily from the fact that the functor from algebraic $N$-modules to torsion $\frakn$-modules, given by differentiating the action, is an equivalence.
	
	Let us remark that the $K$-module structure on $av^M_K (V)$ in this description, as the subspace of $\frakn$-torsion in $V$, is the unique one which when restricted to $M$ gives us the original $M$-module structure, and when restricted to $N$ differentiates to the $\frakn$-module structure.
\end{proof}

\section{Some commutation properties}

Denote $A := \Gamma (\calA)$, the algebra of global sections. We consider $A$ as a $D$-algebra on the point $pt$. We have the localization-globalization adjunction $$ \Lambda : Mod(A , K) \rightleftarrows Mod (\calA , K) : \Gamma.$$ The functor $\Gamma$ is given by taking global sections, while $\Lambda$ is given by tensoring $\calA\otimes_A \cdot$.

\begin{lemma}\label{lem_loc}\
	\begin{enumerate}[(1)]
		\item The functor $Av^M_K$ commutes with localization, i.e. the following diagram naturally $2$-commutes: $$ \xymatrix{Mod(\calA , M) \ar[r]^{Av^M_K} & Mod(\calA , K) \\ Mod(A , M) \ar[r]_{Av^M_K} \ar[u]^{\Lambda} & Mod(A , K) \ar[u]_{\Lambda}} $$
		\item The functor $av^M_K$ commutes with globalization, i.e. the following diagram naturally $2$-commutes: $$ \xymatrix{Mod(\calA , K) \ar[d]_{\Gamma} & Mod(\calA , M) \ar[d]^{\Gamma} \ar[l]_{av^M_K} \\ Mod(A , K) & Mod(A , M) \ar[l]^{av^M_K}} $$
	\end{enumerate}
\end{lemma}

\begin{proof}
	The verification is immediate, by chasing the diagrams.
\end{proof}

Another commutation is of $Av^M_K$ with twist.

\begin{lemma}\label{lem_AvTw}
	Let $\calL$ be a $K$-equivariant line bundle on $Y$. The following diagram naturally $2$-commutes: $$ \xymatrix{Mod(\calA , M) \ar[r]^{Av^M_K} \ar[d]_{\calL \otimes_{\calO} \cdot} & Mod(\calA , K) \ar[d]^{\calL \otimes_{\calO} \cdot} \\ Mod(\calA_{\calL} , M) \ar[r]_{Av^M_K} & Mod(\calA_{\calL} , K) } $$
\end{lemma}

\begin{proof}
	The verification is immediate, by chasing the diagrams.
\end{proof}


\pagebreak
\chapter{The Casselman-Jacquet and Bernstein functors}\label{chp_CJB}

In this chapter we define the Casselman-Jacquet functor $\calC$, the Bernstein functor $\calB$ and the geometric Bernstein functor $\bbB$. We show their finiteness properties, and compare the functor $\calC$ to the "traditional" Jacquet-Casselman functor $J^*$. We also recall how to see that $J^*$ is exact.

We work under the notations and assumptions of section \ref{sec_setting}.

\section{Definition}

We define the Casselman-Jacquet and Bernstein functors using the formalism of chapter \ref{sec_func}. For this, we set $Y = pt$ and $\calA = U(\frakg)$.

\begin{definition}
The \textbf{Bernstein functor} $$ \calB : Mod(\frakg , MN) \to Mod(\frakg, K) $$ is defined as $$ \calB := Av^M_K \circ Rel^{MN}_M.$$ The \textbf{Casselman-Jacquet functor} $$  Mod(\frakg, MN) \leftarrow Mod(\frakg , K) : \calC$$ is defined as $$ \calC := av^M_{MN} \circ rel^K_M.$$
\end{definition}

We also define the geometric Bernstein functor using the formalism of chapter \ref{sec_func}. For this, we set $Y$ to be the flag variety and $\calA = D_{\lambda}$.

\begin{definition}
	The \textbf{geometric Bernstein functor} $$ \bbB : Mod(D_{\lambda} , MN) \to Mod(D_{\lambda} , K)$$ is defined as $$ \bbB := Av^M_K \circ Rel^{MN}_M.$$
\end{definition}

The formulas for $\calB$ and $\bbB$ are only superficially the same, because the space $Y$ is absent from the notation.

The relation between $\calB$ and $\calC$ is as follows.
\begin{proposition}
There is a natural adjunction $$ \calB : Mod(\frakg, MN) \rightleftarrows Mod(\frakg , K) : \calC.$$
\end{proposition}

\begin{proof}
This is clear from the adjunctions of section \ref{sec_rel}.
\end{proof}

The relation between $\calB$ and $\bbB$ is as follows.

\begin{lemma}\label{lem_Bloc}
	The following diagram $2$-commutes:
	$$\xymatrix{Mod(D_{\lambda}, MN) \ar[r]^{\bbB} & Mod(D_{\lambda}, K) \\ Mod^{[\lambda]}(\frakg , MN) \ar[r]_{\calB} \ar[u]^{\Lambda} & Mod^{[\lambda]}(\frakg , K) \ar[u]_{\Lambda}}$$
\end{lemma}

\begin{proof}
	This is a particular case of the more general lemma \ref{lem_loc} (notice that $Rel^{MN}_M$, being the forgetful functor, clearly commutes with localization).
\end{proof}

Let us state another property of $\bbB$:

\begin{lemma}\label{lem_Btrans}
Let $\calL_{\mu}$ be a $G$-equivariant line bundle on $X$. The following diagram $2$-commutes:
$$ \xymatrix{ Mod (D_{\lambda} , MN) \ar[r]^{\bbB} \ar[d]_{\calL_{\mu} \otimes_{\calO} \cdot} & Mod (D_{\lambda} , K) \ar[d]^{\calL_{\mu} \otimes_{\calO} \cdot} \\ Mod (D_{\lambda + \mu} , MN) \ar[r]_{\bbB} & Mod (D_{\lambda + \mu} , K)}$$
\end{lemma}

\begin{proof}
	This is a particular case of the more general lemma \ref{lem_AvTw} (notice that $Rel^{MN}_M$, being the forgetful functor, clearly commutes with twist).
\end{proof}

Finally, let us spell out concretely what are the functors $\calB$ and $\calC$.

Given a $(\frakg , MN)$-module $V$, the $(\frakg , K)$-module $\calB(V)$ is given as $$ \calB (V) = \wedge^{\text{top}}(\frakk / \frakm)^* \otimes Fun_M (K, V)_{\frakk}.$$ Here , $ \wedge^{\text{top}}(\frakk / \frakm)^*$ is a one-dimensional vector space not interfering with the formulas for the action. $Fun_M (K,V)$ is the space of algebraic functions $\phi : K \to V$, satisfying $\phi (km) = m^{-1} \phi (k)$ for $m \in M$ and $k \in K$. The $K$-action is via the left regular action on $K$, and the $\frakg$-action is $(\xi f) (k) = \presuper{k^{-1}}{\xi} \cdot f(k)$ for $\xi \in \frakg , \ k \in K$ and $f: K \to V$. The $\frakk$-action w.r.t. which we take the coinvariants is the difference between the $\frakk$-action gotten by differentiating the $K$-action, and the $\frakk$-action gotten by restricting the $\frakg$-action.

Given  a $(\frakg , K)$-module $V$, the $(\frakg , MN)$-module $\calC(V)$ is given as $$ \calC (V) = \wedge^{\text{top}}(\frakk / \frakm) \otimes \left( \prod_{\alpha} V_{\alpha} \right)^{M-\text{fin.}, \ \frakn-\text{tor.}}. $$ Here again $ \wedge^{\text{top}}(\frakk / \frakm)$ is a one-dimensional vector space not interfering with the formulas for the action. The product is the product of $K$-types of $V$, and it carries a $(\frakg , M)$-action "type-wise". We then take the submodule of $M$-finite and $\frakn$-torsion vectors.

\section{Finiteness properties}

\begin{lemma}\label{lem_Bcoh}
	The functor $\bbB$ sends $\calM(D_{\lambda},MN)$ into $\calM(D_{\lambda}, K)$.
\end{lemma}

\begin{proof}
	Recall that, since $MN$ (resp. $K$) has finitely many orbits on $X$, coherent $MN$-equivariant (resp. $K$-equivariant) $D_{\lambda}$-modules are the same as holonomic $MN$-equivariant (resp. $K$-equivariant) $D_{\lambda}$-modules. The lemma thus follows from the fact that holonomicity is preserved under all the performed operations in the definition of $\bbB$.
\end{proof}

\begin{claim}
	The functor $\calB$ preserves infinitesimal character\footnote{this means that if a module is killed by some ideal in $\calZ(\frakg)$, then $\calB$ applied to this module is killed by the same ideal as well.} and sends $\calM(\frakg , MN)$ into $\calM(\frakg , K)$.
\end{claim}

\begin{proof}

That $\calB$ preserves infinitesimal character is clear from its formula. Thus it is enough to show that for $V \in \calM(\frakg , MN)$, the module $\calB(V)$ is finitely-generated over $U(\frakg)$. By filtering $V$ we can assume, without loss of generality, that $V \in \calM^{\theta}(\frakg , MN)$ for some infinitesimal character $\theta$.

Let $\lambda \in \wt$ be an antidominant weight satisfying $[\lambda] = \theta$, and consider the localization adjunction $$ \Lambda : Mod^{\theta}(\frakg , \cdot)  \rightleftarrows Mod (D_{\lambda} , \cdot): \Gamma.$$ By lemma \ref{lem_Bloc}, we have $\Lambda (\calB (V)) \is \bbB(\Lambda (V))$. Since $\lambda$ is antidominant, we have $\Gamma \circ \Lambda \is id$, so that we get $\calB (V) \is \Gamma (\bbB(\Lambda (V))).$ Now, $\Lambda (V)$ is coherent, since $V$ is finitely-generated. And since $\bbB$ preserves coherence by lemma \ref{lem_Bcoh}, $\bbB(\Lambda (V))$ is coherent. Since $\Gamma$ sends coherent $D_{\lambda}$-modules to finitely-generated $U(\frakg)$-modules (see \cite[Section L.1, Lemma 23]{M}), we obtain that $\calB (V) \is \Gamma (\bbB(\Lambda (V)))$ is finitely-generated as a $U(\frakg)$-module.

\end{proof}

\begin{claim}
	The functor $\calC$ preserves infinitesimal character and sends $\calM(\frakg , K)$ into $\calM(\frakg , MN)$.
\end{claim}

\begin{proof}

That $\calC$ preserves infinitesimal character is clear from its formula. Thus it is enough to show that for $V \in \calM(\frakg , K)$, the module $\calC(V)$ is finitely-generated over $U(\frakg)$.

By the lemma following this claim (lemma \ref{lem_MNcat}) it is enough to show that $\calC(V)^{\frakn^m}$ are finite-dimensional, where $V \in \calM(\frakg , K)$, and $(\bullet)^{\frakn^m} := \{ w \in \bullet \ | \ \frakn^m w = 0\}$. Proposition \ref{prop_relJC}, that we will discuss in the next subsection, expresses $\calC(V)$ as $J^*(V^{\vee})$ (times a factor which does not matter). Writing $W := V^{\vee}$, $J^*(W)^{\frakn^m}$ consists of functionals on $W$ which kill $\frakn^m W$. Since $W$, being a module in $\calM (\frakg , K)$, is finitely generated as an $U(\frakn)$-module, $\frakn^m W$ is of finite codimension in $W$, and this shows that $J^*(W)^{\frakn^m}$ is finite-dimensional.

\end{proof}

We needed the following lemma in the previous claim:

\begin{lemma}\label{lem_MNcat}
Let $V$ be a $U(\frakg)$-module for which: (1) $V = \cup_{m \ge 1} V^{\frakn^m}$ where $V^{\frakn^m} = \{ v \in V \ | \ \frakn^m v = 0 \}$ (2) $V^{\frakn^m}$ is finite-dimensional for every $m \ge 1$ and (3) $\calZ(\frakg)$ acts on $V$ via a finite-dimensional quotient. Then $V$ is finitely-generated as an $U(\frakg)$-module.
\end{lemma}

\begin{proof} Denote by $V^{\omega} \subset V$ the generalized weight space of $\frakt$ with weight $\omega \in \frakt^*$. Since $\frakt$ stabilizes the $V^{\frakn^m}$'s, and in view of (1), $V = \oplus_{\omega} V^{\omega}$. Furthermore, by the Harish-Chandra isomorphism, and in view of (3), the weights by which $\frakt$ acts on $V^{\frakn}$ are all contained in a finite subset $S \subset \frakt^*$. This forces that for any fixed $\omega \in \frakt^*$, there exists $m \ge 1$ such that $V^{\omega} \subset V^{\frakn^m}$ (because weights appearing in $V^{\frakn^{m+1}} / V^{\frakn^m}$ become further and further away from $S$, as $m$ grows). From this we conclude that the $V^{\omega}$'s are all finite-dimensional.

Denote now by $W$ the sub-$U(\frakg)$-module of $V$ generated by $\sum_{\omega \in S} V^{\omega}$. $W$ is a finitely generated $U(\frakg)$-module, and we would like to show that $V=W$, thus concluding the proof.

To that end, consider $U:= V / W$. As a quotient of $V$, $U$ is also the sum of finite-dimensional generalized weight spaces of $\frakt$. Since taking generalized weight space is an exact functor, $U^{\omega} = 0$ for $\omega \in S$. But this forces $U^{\frakn} = 0$, since again all weights by which $\frakt$ acts on $U^{\frakn}$ are contained in $S$. Now, $U^{\frakn} = 0$ forces $U = 0$ (since, would we have a non-zero vector in $U$, we would start applying elements from $\frakn$ to it; Eventually it will become zero, and at one step before that, we will get a non-zero element of $U^{\frakn}$). Thus we have $V=W$ and we are done.

\end{proof}

From now on, we will consider the functors $\calC , \calB , \bbB$ from, and to, the categories $\calM(\cdot)$ rather than $Mod(\cdot)$.

\section{Comparison with the traditional Casselman-Jacquet functor}

Let us describe how the functor $$\calC : \calM (\frakg , MN) \to \calM(\frakg , K)$$ is related to the "usual Jacquet-Casselman" functor.

First, let us recall the "usual Jacquet-Casselman" functor, in its contravariant form (call it $J^*$): For a module $V \in \calM(\frakg , K)$, $J^*(V)$ is the submodule of $\frakn$-torsion vectors in the abstract dual of $V$.

Recall also the duality equivalence $(\cdot)^{\vee} : \calM (\frakg , K)^{op} \to \calM (\frakg , K)$, given by sending a module $V = \oplus V_{\alpha}$ to $V^{\vee} := \oplus V_{\alpha}^*$ (here, $V_{\alpha}$ denote the $K$-types of $V$).

\begin{proposition}\label{prop_relJC}
$\calC \is \wedge^{\text{top}}(\frakk / \frakm) \otimes J^* \circ (\cdot)^{\vee}$ (as functors $\calM(\frakg , K) \to \calM(\frakg , MN)$).
\end{proposition}

\begin{proof}
This is immediate. Clearly, $\prod V_{\alpha}$ can be identified with the abstract dual of $V^{\vee}$. Thus, $\calC(V)$ can be identified with the $\frakn$-torsion, $M$-finite vectors in the abstract dual of $V^{\vee}$. Since the subspaces killed by $\frakn^m$ are finite-dimensional and preserved by $M$, one actually can drop the demand "$M$-finite".
\end{proof}

Casselman shows that $J^*$ is exact. We will briefly recall how to show it in subsection \ref{sec_Jex}.

\begin{theorem}[Casselman]\label{thm_Cex}
	The functor $\calC: \calM(\frakg , K) \to \calM(\frakg , MN)$ is exact.
\end{theorem}

\begin{proof}
	This immediately follows from $J^*$ being exact, and proposition \ref{prop_relJC}.
\end{proof}

\subsection{The exactness of $J^*$}\label{sec_Jex}

Let us sketch, following Casselman, how to prove the exactness of $J^*$ mentioned above. In fact, one shows the exactness of a more general functor. Let $\frakn$ be a nilpotent (finite-dimensional) Lie algebra over $k$. For an $\frakn$-module $V$ and $m\ge 1$, we denote $V^{\frakn^m} := \{ v\in V \ | \ \frakn^m v = 0 \}$ and $V_{\frakn^m} := V / \frakn^m V$.

We consider the categories:

\begin{enumerate}[(1)]
\item $\calM(\frakn)$ - the category of finitely-generated $U(\frakn)$-modules.
\item $\calM_{\text{tor}}(\frakn)$ - the category of $U(\frakn)$-modules $V$ which satisfy $V = \varinjlim V^{\frakn^m}$ and that $V^{\frakn}$ is finite-dimensional as a $k$-vector space.
\item $\calM_{\text{com}}(\frakn)$ - the category of $U(\frakn)$-modules $V$ which satisfy $V = \varprojlim V_{\frakn^m}$ and that $V_{\frakn}$ is finite-dimensional as a $k$-vector space.
\end{enumerate}

Also, we consider the functors:

\begin{enumerate}[(1)]
\item $J^* : \calM(\frakn)^{op} \to \calM_{\text{tor}} (\frakn)$ sending a module $V$ to $\varinjlim (V^*)^{\frakn^m}$.
\item $\hat{J} : \calM(\frakn) \to \calM_{\text{com}} (\frakn)$ sending a module $V$ to $\varprojlim V_{\frakn^m}$.
\end{enumerate}

And, we consider the duality functors:

\begin{enumerate}[(1)]
\item $(\cdot)^{\vee} : \calM_{\text{tor}} (\frakn)^{op} \to \calM_{\text{com}} (\frakn)$ sending a module $V$ to $V^*$.
\item $(\cdot)^{\vee} : \calM_{\text{com}} (\frakn)^{op} \to \calM_{\text{tor}} (\frakn)$ sending a module $V$ to $V^{\text{cont}-*}$, the continuous dual; That is, $V^{\text{cont}-*}$ consists of functionals on $V$, which annihilate $\frakn^m V$ for some $m \ge 1$.
\end{enumerate}

\begin{lemma}\
\begin{enumerate}
\item The dualities $(\cdot)^{\vee}$ are well-defined, and are mutually inverse equivalences of categories.
\item One has a canonical isomorphism $J^* \is (\cdot)^{\vee} \circ \hat{J}$.
\end{enumerate}
\end{lemma}

\begin{proposition}
The functor $\hat{J}$ is exact.
\end{proposition}

\begin{proof}
This is done as in the classical Artin-Rees approach.

Step 1: It is enough to show that for $U \subset V \in \calM(\frakn)$, there exists $d \ge 0$ such that $\frakn^{m+d} V \cap U \subset \frakn^m U$ for large $m$.

Step 2: To show what is wanted in step 1, it is enough to show that the algebra $U(\frakn)^{\star}$ is left Noetherian; Here, $U(\frakn)^{\star}$ is $ \oplus_{m\ge 0} \frakn^m U(\frakn)$. Indeed, once we know that $U(\frakn)^{\star}$ is left Noetherian, we consider (with the notation of step 1) the $U(\frakn)^{\star}$-module $U_V^{\star} := \oplus_{m\ge 0} (\frakn^m V \cap U)$. It is finitely generated being a submodule of $V^{\star} := \oplus_{m\ge 0} \frakn^m V$. Writing what the finite-generation means, immediately gives what is wanted in step 1.

Step 3: (taken from \cite[section 0.6.4]{W}) To show that $U(\frakn)^{\star}$ is left Noetherian, consider $\tld{\frakn} = \oplus_{m\ge 0} D^m \frakn \subset U(\frakn)^{\star}$. Here, $D^0 \frakn := \frakn , D^1 \frakn := \frakn$, and $D^m \frakn = [ D^{m-1} \frakn , \frakn]$ for $m \ge 2$. Now, $\tld{\frakn}$ is a Lie algebra inside $U(\frakn)^{\star}$, and thus we get a morphism $U(\tld{\frakn}) \to U(\frakn)^{\star}$. This morphism is surjective (since already $\frakn$ in degrees $0$ and $1$ suffice to generate $U(\frakn)^{\star}$). Now note, crucially, that $\tld{\frakn}$ is finite-dimensional, since $\frakn$ is nilpotent! Hence $U(\tld{\frakn})$ is left Noetherian (as the universal enveloping algebra of a finite-dimensional Lie algebra), and thus $U(\frakn)^{\star}$ is left Noetherian as a quotient of $U(\tld{\frakn})$.

\end{proof}

\begin{corollary}
The functor $J^*$ is exact.
\end{corollary}

\begin{proof}
This follows from the proposition and $J^* \is (\cdot)^{\vee} \circ \hat{J}$.
\end{proof}


\pagebreak
\chapter{The derived geometric Bernstein functor}

In this chapter we define the derived geometric Bernstein functor $\bbB^{der}$. When composing it with $0$-th cohomology, we reobtain $\bbB$. We prove a basic (for us) property of $\bbB^{der}$, that is sends standard objects to acyclic objects.

Let us mention that although in the technical approach of this thesis we only use $\bbB^{der}$ to obtain properties of $\bbB$, a better approach would be to consider $\bbB^{der}$ as the "main player". One would like then to show that $\bbB^{der}$ admits a right adjoint. See conjectures \ref{cnj_gbfa} and \ref{cnj_gjf}.

\section{Equivariant derived categories}

Let $Y$ be a smooth algebraic variety, equipped with an action of an affine algebraic group $H$. Let $\calA$ be an $H$-equivariant TDO on $Y$. We can consider the abelian category $\calM (\calA , H)$ of holonomic $H$-equivariant $\calA$-modules on $Y$. Bernstein and Lunts (\cite{BL}) explain how to define a bounded $t$-category $\calD (\calA , H)$, with heart $\calM (\calA , H)$ - the derived category of $H$-equivariant holonomic $\calA$-modules on $Y$. It is not, in general, equivalent to the naive bounded derived category of $\calM (\calA , H)$.

We briefly recall the definition of $\calD (\calA , H)$. Consider the category of smooth $H$-resolutions $Res(Y)$; Namely, an object of $Res(Y)$ is a free $H$-variety $Z$ together with a smooth $H$-morphism $\pi : Z \to Y$. Given such a resolution, denoting by $q : Z \to H \backslash Z$ the quotient map, on $H \backslash Z$ we have the TDO $q_+ \pi^{\cdot}\calA$ associated to the $H$-equivariant TDO $\pi^{\cdot}\calA$ on $Z$. We associate to our resolution the derived category of holonomic $q_+ \pi^{\cdot}\calA$-modules on $H \backslash Z$, of which we think as an approximation to $\calD (\calA , H)$. This association is naturally a $2$-functor, and we take its $2$-limit. This is $\calD (\calA , H)$.

There is always a canonical $t$-exact functor $D^b (\calM (\calA , H)) \to \calD (\calA , H)$, identifying the hearts.

\begin{remark}
We considered holonomic $\calA$-modules just because this is the variant that we will consider below. One can define in the same way as above the derived category of $H$-equivariant $\calA$-modules.
\end{remark}

\section{Definition}

Consider again the diagram of chapter \ref{sec_func}: $$ \xymatrix{ K \times X \ar[r]^{\widetilde{a}} \ar[dd]_{p} & K \times X \ar[d]^{q} & \\ & K /M \times X \ar[d]^{\widetilde{p}} \\ X & X }$$ and recall the equivariant structures that we imposed there on the varieties in the diagram, as well as the TDO's (we take the TDO $D_{\lambda}$ on $X$ and on the other spaces we take its pullback along the vertical maps).

We wish to define a functor $$\bbB^{der} : \calD(D_{\lambda} , MN) \to \calD (D_{\lambda} , K)$$ such that $H^0 \bbB^{der} \is \bbB$.

\begin{definition}
The \textbf{derived geometric Bernstein functor} $$\bbB^{der} : \calD (D_{\lambda} , MN) \to \calD (D_{\lambda} , K)$$ is defined as the composition $$ \xymatrix{ & \calD (p^{\cdot} D_{\lambda},K \times M) \ar[r]^{\widetilde{a}_* } & \calD (p^{\cdot} D_{\lambda},K \times M) \ar[d]^{q_+} \\ & & \calD (\widetilde{p}^{\cdot} D_{\lambda},K) \ar[d]^{\widetilde{p}_*} \\ \calD (D_{\lambda},MN) \ar[r]^{for} & \calD (D_{\lambda},M) \ar[uu]^{p^{\circ}} & \calD (D_{\lambda},K) }$$
\end{definition}

In the above definition, the functor $for$ is the ($t$-exact) forgetful functor. The functor $p^{\circ}$ is the $t$-exact pullback (it is an induction equivalence). The functor $\widetilde{a}_*$ is the ($t$-exact) equivalence of "transport of structure" along the isomorphism $\widetilde{a}$. The functor $q_+$ is the ($t$-exact) "quotient equivalence". The functor $\widetilde{p}_*$ is the pushforward - it is right-$t$-exact.

The functor $\bbB^{der}$ is right-$t$-exact.

\begin{lemma}
	The functors $$H^0 \bbB^{der} , \bbB : \calM( D_{\lambda} , MN) \to \calM(D_{\lambda} , K)$$ are naturally isomorphic.
\end{lemma}

\begin{proof}
	This is clear from inspecting the operations in the definition of both functors.
\end{proof}

\section{Commutation properties}

\begin{lemma}\label{lem_BerTrans}
Let $\calL_{\mu}$ be a $G$-equivariant line bundle on $X$. The following diagram $2$-commutes:
$$ \xymatrix{ \calD (D_{\lambda} , MN) \ar[r]^{\bbB^{der}} \ar[d]_{\calL_{\mu} \otimes_{\calO} \cdot} & \calD (D_{\lambda} , K) \ar[d]^{\calL_{\mu} \otimes_{\calO} \cdot} \\ \calD (D_{\lambda+\mu} , MN) \ar[r]_{\bbB^{der}} & \calD (D_{\lambda+\mu} , K)}$$
\end{lemma}

\begin{proof}
This follows from the projection formula.
\end{proof}

For the next commutation property, we recall the intertwining functors. Namely, fixing a closed subgroup $H \subset G$, for every $w \in W$ we have functors $I_{*,w} , I_{!,w} : \calD(D_{\lambda}, H) \to \calD(D_{w \lambda}, H)$. These are defined by considering the diagram $$ \xymatrix{& Z_w \ar[rd]^{p_2} \ar[ld]_{p_1} & \\ X & & X} $$ where $Z_w \subset X\times X$ is the subvariety of pairs of points in relative position $w$, and $p_1,p_2$ are the projections on the first and second factor. One then defines $$I_{*,w}(\calF) = (p_2)_* (\calL_w \otimes_{\calO} p_1^{\circ}(\calF)) \quad , \quad I_{!,w}(\calF) = (p_2)_! (\calL_w \otimes_{\calO} p_1^{\circ}(\calF)).$$ Here, $\calL_w$ is a suitable $G$-equivariant line bundle, fixing the discrepancy between the TDO's $p_1^{\cdot} D_{\lambda}$ and $p_2^{\cdot} D_{w \lambda}$.

One has the following properties of the intertwining functors: First of all, if $\lambda$ is antidominant and regular, $$I_{*,w} : \calD(D_{\lambda}) \to \calD(D_{w \lambda})$$ commutes with the derived global sections functors (here $H=1$, because we don't want to think about derived global sections in the equivariant setting). Second, the functor $$I_{!,w^{-1}} : \calD(D_{w \lambda}, H) \to \calD(D_{\lambda}, H)$$ is inverse to the functor $$I_{*,w} : \calD(D_{\lambda}, H) \to \calD(D_{w \lambda}, H).$$

\begin{lemma}\label{lem_BerInt}
For every $w \in W$, the following diagram $2$-commutes:
$$ \xymatrix{ \calD (D_{\lambda} , MN) \ar[r]^{\bbB^{der}} \ar[d]_{I_{*,w}} & \calD (D_{\lambda} , K) \ar[d]^{I_{*,w}} \\ \calD (D_{w \lambda} , MN) \ar[r]_{\bbB^{der}} & \calD (D_{w \lambda} , K)}$$
\end{lemma}

\begin{proof}
This follows from smooth base change, and the projection formula.
\end{proof}

\section{The category $\calD (D_{\lambda} , MN)$}\label{sec_MN}

We will use, regarding the category $\calD (D_{\lambda} , MN)$, the definitions and statements of appendix \ref{app_str}.

The Weyl group $W$ has the Bruhat order, and we call a subset $I \subset W$ closed, if $w\in I$ and $v \leq w$ imply $v \in I$. Closed subsets of $W$ are in bijection with $MN$-invariant closed subvarieties of $X$ - associating to $I \subset W$ the subvariety $X_I = \cup_{w \in I} \calO_w$. We associate to a closed subset $I \subset W$ the Serre subcategory of $\calM (D_{\lambda}, MN)$, which consists of modules whose support lies in $X_I$.

\begin{claim}
	The above association defines an affine simple $W$-stratification on $\calM (D_{\lambda}, MN)$, and $D^b (\calM (D_{\lambda}, MN)) \to \calD (D_{\lambda}, MN)$ is an equivalence.
\end{claim}

\begin{proof}
	This is standard, using the criteria of appendix \ref{app_str}.
\end{proof}

The above claim makes $\calM (D_{\lambda}, MN)$ very comfortable to work with; It has enough projectives and injectives, is of finite cohomological dimension, has the formalism of standard and costandard objects, etc.; See appendix \ref{app_str}.

\section{An acyclicity property}

\begin{proposition}\label{prop_BDerTilt}
	The functor $$ \bbB^{der} : \calD(D_{\lambda} , MN) \to \calD(D_{\lambda} , K)$$ sends standard-filtered objects to acyclic objects.
\end{proposition}

\begin{proof}
	
	We reduce first to the case when $\lambda$ is antidominant and regular. This is done by choosing a $G$-equivariant line bundle $\calL_{\mu}$ such that $\lambda + \mu$ is antidominant and regular, and using lemma \ref{lem_BerTrans} (notice that tensoring with $\calL_\mu$ is a $t$-exact equivalence, which also preserves the $W$-stratification structure on the $MN$-equivariant side).
	
	So, let us assume that $\lambda$ is antidominant and regular. It is enough to check that $\bbB^{der}$ sends a standard object to an acyclic object, so fix a standard object $(i_w)_! [E]_w^{\lambda}$. To check that $\bbB^{der} ((i_w)_! [E]_w^{\lambda})$ is acyclic, it is enough to do so when forgetting the $K$-equivariancy. Furthermore, since $\lambda$ is antidominant and regular, it is enough to check that the derived global sections of $\bbB^{der} ((i_w)_! [E]_w^{\lambda})$ are acyclic. Those are isomorphic to the derived global sections of $I_{*,w^{-1}} \bbB^{der} ((i_w)_! [E]_w^{\lambda})$. We have: $$ I_{*,w^{-1}} \bbB^{der} ((i_w)_! [E]_w^{\lambda}) \is \bbB^{der} (I_{*,w^{-1}} (i_w)_! [E]_w^{\lambda}) \is \bbB^{der} ((i_1)_! [E]_1^{w^{-1} \lambda}) \is j_* [E]^{w^{-1} \lambda}.$$
	
	Here, the first isomorphism is by lemma \ref{lem_BerInt}. For the second one, notice that $I_{!,w} (i_1)_! [E]_1^{w^{-1} \lambda} \is (i_w)_! [E]_w^{\lambda}$ (this immediately follows from the geometry in the definition of $I_{!,w}$), and recall that $I_{*,w^{-1}}$ is inverse to $I_{!,w}$, so that $ I_{*,w^{-1}} (i_w)_! [E]_w^{\lambda} \is [E]_1^{w^{-1} \lambda}$. For the third isomorphism notice that $\bbB^{der} ([E]_1^{w^{-1} \lambda}) \is j_* ([E]^{w^{-1} \lambda})$ - this calculation again follows immediately from the geometry in the definition of $\bbB^{der}$.
	
	So, we are interested to show that the derived global sections of $j_* ([E]^{w^{-1} \lambda})$ are acyclic. This is the case because $U$ is affine, and hence $j: U \to X$ is an affine morphism.
	
\end{proof}


\pagebreak
\chapter{The functor $\bbC$ and its property}\label{chp_bbC}

In this chapter we show that $\bbB$ admits a right adjoint $\bbC$, and show its basic property - that it sends injective objects to costandard-filtered objects.

\section{The functor $\bbC$}

\begin{claim}\label{clm_bbC}
	The functor $$\bbB : \calM(D_{\lambda} , MN) \to \calM(D_{\lambda} , K)$$ admits an exact right adjoint $$ \calM(D_{\lambda} , MN) \leftarrow \calM(D_{\lambda} , K) : \bbC.$$
\end{claim}

\begin{proof}
	First assume that $\lambda$ is antidominant and regular. Then in the $2$-commutative diagram $$\xymatrix{\calM(D_{\lambda}, MN) \ar[r]^{\bbB} & \calM(D_{\lambda}, K) \\ \calM^{[\lambda]}(\frakg , MN) \ar[r]_{\calB} \ar[u]^{\Lambda} & \calM^{[\lambda]}(\frakg , K) \ar[u]_{\Lambda}}$$ the vertical functors are equivalences. Hence, $\bbB$ admits an exact right adjoint, since $\calB$ does.
	
	In the case of a general $\lambda$, choose a $G$-equivariant line bundle $\calL_{\mu}$ such that $\lambda + \mu$ is antidominant and regular. Then the $2$-commutative diagram $$ \xymatrix{ \calM (D_{\lambda} , MN) \ar[r]^{\bbB} \ar[d]_{\calL_{\mu} \otimes_{\calO} \cdot} & \calM (D_{\lambda} , K) \ar[d]^{\calL_{\mu} \otimes_{\calO} \cdot} \\ \calM (D_{\lambda + \mu} , MN) \ar[r]_{\bbB} & \calM (D_{\lambda + \mu} , K)}$$ shows that $\bbB$ between the weight $\lambda$ categories admits an exact right adjoint, since $\bbB$ between the weight $\lambda + \mu$ categories admits an exact right adjoint (as the vertical functors are equivalences).
\end{proof}

\begin{lemma}\label{prop_CasGamma}\
The following diagram $2$-commutes:
$$\xymatrix{\calM(D_{\lambda}, MN) \ar[d]_{\Gamma} & \calM(D_{\lambda}, K) \ar[l]_{\bbC} \ar[d]^{\Gamma} \\ \calM^{[\lambda]}(\frakg , MN) & \calM^{[\lambda]}(\frakg , K) \ar[l]^{\calC}}$$
\end{lemma}

\begin{proof}
This immediately follows from the commutation of $\bbB$ with $\Lambda$, by taking adjoints.
\end{proof}

\section{Acycilicity properties}

Since $\calM (D_{\lambda} , MN)$ admits enough projective objects and is of finite cohomological dimension, we can construct the left derived functor of $\bbB$: $$ L\bbB : D^b (\calM(D_{\lambda}, MN)) \to D^b (\calM(D_{\lambda}, K)).$$

For the notions in the next lemma and its proof, see definition \ref{def_BigEnough} and the remark that follows it.

\begin{lemma}\label{lem_StAdapt}
The class of standard-filtered objects in $\calM(D_{\lambda}, MN)$ is big enough and adapted to $\bbB$.
\end{lemma}

\begin{proof}
The class of standard-filtered objects in $\calM (D_{\lambda} , MN)$ is big enough by lemma  \ref{lem_BigEnough}. Thus, it is enough to see that this class is adapted to $\bbB = H^0 \bbB^{der}$. Thus, let $$0 \to \calF_1 \to \calF_2 \to \calF_3 \to 0$$ be a short exact sequence of standard-filtered objects in $\calM (D_{\lambda} , MN)$. Applying the functor $\bbB^{der}$ and taking cohomology, we obtain the exact sequence $$ 0 \to H^0 \bbB^{der} \calF_1 \to H^0 \bbB^{der} \calF_2 \to H^0 \bbB^{der} \calF_3 \to H^1 \bbB^{der} \calF_1.$$ By proposition \ref{prop_BDerTilt}, $H^1 \bbB^{der} \calF_1 = 0$. Hence, the exact sequence becomes $$ 0 \to H^0 \bbB^{der} \calF_1 \to H^0 \bbB^{der} \calF_2 \to H^0 \bbB^{der} \calF_3 \to 0$$ and indeed the class of standard-filtered objects is adapted to $\bbB$.
\end{proof}

\begin{proposition}\label{prop_LBtilt}
The functor $$ L\bbB : D^b(\calM(D_{\lambda} , MN)) \to D^b(\calM(D_{\lambda} , K))$$ sends standard-filtered objects to acyclic objects.
\end{proposition}

\begin{proof}
This follows immediately from lemma \ref{lem_StAdapt}, since left derived functors can be computed term-wise on complexes of objects in a big enough adapted class.
\end{proof}

Since the left-exact functor $\bbC$ is in fact exact, its right-derived functor $$ D^b(\calM(D_{\lambda}, MN)) \leftarrow D^b(\calM(D_{\lambda}, K)): R\bbC$$ is constructed trivially (applying the functor to complexes term-wise). The functor $L\bbB$ is left adjoint to the functor $R\bbC$.

\begin{proposition}\label{prop_Ctilt}
The functor $$ \calM(D_{\lambda} , MN) \leftarrow \calM(D_{\lambda} , K) : \bbC$$ sends injective objects to costandard-filtered objects.
\end{proposition}

\begin{proof}
Let $I \in \calM (D_{\lambda} , K)$ be injective. In view of lemma \ref{lem_charstflt}, we want to show that for every standard object $S \in \calM (D_{\lambda} , MN)$, The $Hom$-space $RHom_{D^b(\calM (D_{\lambda} , MN))}(S, \bbC I)$ is acyclic. By adjunction, this is the same as $RHom_{D^b(\calM(D_{\lambda},K))}(L\bbB S , I)$. Since $L\bbB S$ is acyclic by proposition \ref{prop_LBtilt} and $I$ is injective, this $Hom$-space is clearly acyclic.
\end{proof}


\pagebreak
\chapter{The main statement}\label{chp_main}

In this chapter we prove our main statement, that the module one obtains by applying $\calC$ to a principal series Harish-Chandra module admits a canonical $W$-filtration whose subquotients are twisted Verma modules.

\section{Principal series modules}

For $V \in \calM^{[\lambda]}(\frakg , K)$, denote by $V_{(\frakb , \lambda)} \in \calM^{\lambda}(\frakt , M)$ the module of coinvariants, i.e. the quotient of $V$ by the subspace generated $\xi - (\lambda-\rho)(\xi)$, for $\xi \in \frakb$.

\begin{definition}\label{def_PrinSer}
\begin{enumerate}[(1)]
	\item  the \textbf{principal series functor} $$ \calM^{[\lambda]}(\frakg , K) \leftarrow \calM^{\lambda} (\frakt , M) : Pr^{\lambda}$$ is the right adjoint to $$ (\cdot)_{(\frakb , \lambda)} : \calM^{[\lambda]}(\frakg , K) \to \calM^{\lambda} (\frakt , M).$$
	\item The \textbf{geometric principal series functor}  $$ \calM (D_{\lambda} , K) \leftarrow \calM^{\lambda} (\frakt , M) : \calP r^{\lambda}$$ is defined as $\calP r^{\lambda} := j_* \circ [\cdot]^{\lambda}$.
\end{enumerate}

\end{definition}

\begin{remark}
	We have $\calP r^{\lambda} \is \bbB \circ (i_1)_! \circ [\cdot]^{\lambda}_1$; This is clear from the geometry in the definition of $\bbB$.
\end{remark}

The next lemma will, in particular, show that $Pr^{\lambda}$ exists.

\begin{lemma}
We have $Pr^{\lambda} = \Gamma \circ\calP r^{\lambda}$.
\end{lemma}

\begin{proof}

The left adjoint of $\Gamma \circ \calP r^{\lambda} = \Gamma \circ j_* \circ [\cdot]^{\lambda}$ is $ Fib_{x_0} \circ j^* \circ \Lambda$, i.e. it takes a module $V$ to the fiber of $\Lambda (V)$ at $x_0$, which is indeed $V_{(\frakb , \lambda)}$.
\end{proof}

Another property which we record is:

\begin{lemma}\label{lem_PrinInj}
The objects in the image of $$ \calM(D_{\lambda}, K) \leftarrow \calM^{\lambda}(\frakt , M): \calP r^{\lambda}$$ are injective.
\end{lemma}

\begin{proof}
The follows at once from $j_*$ having an exact left adjoint, and $\calM^{\lambda}(\frakt , M)$ being semisimple.
\end{proof}

\section{Twisted Verma modules}\label{sec_tvm}

\begin{definition}
	Given $w \in W$, the \textbf{twisted Verma functor} $$ \calM (\frakg , MN) \leftarrow \calM^{\lambda} (\frakt , M):  V^{\lambda}_w$$ is defined as $$ V^{\lambda}_w := \Gamma \circ (i_w)_* \circ [\cdot]^{\lambda}_w.$$
\end{definition}

\section{Main statement}

\begin{theorem}\label{thm_main}
Let $E \in \calM^{\lambda} (\frakt , M)$.
\begin{enumerate}[(a)]
	\item The object $$\bbC(\calP r^{\lambda} (E)) \in \calM(D_{\lambda} , MN)$$ is costandard-filtered.
	\item The $w$-th subquotient in the canonical $W$-filtration of $\bbC(\calP r^{\lambda} (E))$ is isomorphic to $(i_w)_* [E]^{\lambda}_w$.
\end{enumerate}
\end{theorem}

\begin{remark}
	Although the $W$-filtrations that we obtain in the above theorem and the corollary that follows are canonical, we are not able yet to identify canonically the subquotients, although it is clear that conjecture \ref{cnj_geomlem} should hold.
\end{remark}

\begin{proof}[Proof (of Theorem \ref{thm_main})]
That $\bbC(\calP r^{\lambda} (E))$ is costandard-filtered follows from the facts that objects in the image of  $\calP r^{\lambda}$ are injective (lemma \ref{lem_PrinInj}) and $\bbC$ sends injective objects to costandard-filtered objects (proposition \ref{prop_Ctilt}).

The proof of part $(b)$ will be sketched in appendix \ref{app_detail}.

\end{proof}

\begin{corollary}\label{cor_main}
Let $E \in \calM^{\lambda} (\frakt , M)$. The object  $$ \calC (Pr^{\lambda} (E)) \in \calM (\frakg, MN)$$ (this is the Casselman-Jacquet functor applied to a principal series module) admits a canonical $W$-filtration, whose $w$-th subquotient is isomorphic to $V^{\lambda}_w (E)$.
\end{corollary}

\begin{proof}
We have: $$\calC(Pr^{\lambda}(E)) \is \calC(\Gamma (\calP r^{\lambda}(E))) \is \Gamma (\bbC(\calP r^{\lambda}(E))).$$ We now apply $\Gamma$ to the canonical $W$-filtration on $\bbC(\calP r^{\lambda}(E))$ to obtain a $W$-filtration on $\calC(Pr^{\lambda}(E))$, whose $w$-th subquotient is isomorphic to $\Gamma ( (i_w)_* [E]^{\lambda}_w) = V^{\lambda}_w (E)$. For this to be possible, according to remark \ref{rem_funcfilt}, we want to see that $\Gamma$ transforms a short exact sequence whose first term is costandard-filtered, into a short exact sequence. This is evident, since the derived global sections of an object of the form $(i_w)_* \calF$, where $\calF \in \calM (i_w^{\cdot} D_{\lambda} , MN)$, are acyclic, since $\calO_w$ is affine.
\end{proof}


\pagebreak
\chapter{Some questions and conjectures}\label{chp_conj}

\section{Regarding $\calC$}

The answer to the following question if probably "yes".

\begin{question}\label{que_fg}
Does the functor $$Mod (\frakg , MN) \leftarrow Mod (\frakg , K) : \calC$$ preserve finite-generation over $U(\frakg)$?
\end{question}

The following conjecture is very interesting.

\begin{conjecture}\label{conj_exact}
The functor $\calC$ is exact on the subcategory of modules in $Mod(\frakg , K)$ which are finitely generated over $U(\frakg)$.
\end{conjecture}

\begin{remark}
	Granted that the conjecture above is proven, it makes sense to redefine $\calC$ on the whole of $Mod(\frakg,K)$, as being as before on modules finitely-generated over $U(\frakg)$, and extended to all modules by the request of commuting with direct limits.
\end{remark}

For the next conjecture, recall the duality equivalence $$(\cdot)^{\vee} : \calM(\frakg , MN)^{op} \to \calM(\frakg , M \theta (N))$$ given by taking the $\theta(\frakn)$-torsion vectors in the abstract dual. Also, denote by $\calC^{-}$ the Casselman-Jacquet functor resulting from the initial choice of the Borel $\theta(B)$, rather than $B$.

\begin{conjecture}\label{conj_dual}

The following diagram canonically $2$-commutes, perhaps up to tensoring by some power of $\wedge^{top} (\frakk / \frakm)$:

$$ \xymatrix{\calM(\frakg , MN) \ar[d]_{(\cdot)^{\vee}} & \calM(\frakg , K) \ar[l]_{\calC} \ar[d]^{(\cdot)^{\vee}} \\ \calM(\frakg , M\theta({N})) & \calM(\frakg , K) \ar[l]_{\calC^{-}}} $$

\end{conjecture}

\begin{remark}
	It seems that the above conjecture is known to be true in the analytic case; See the last paragraph in section 12 of \cite{C2}.
\end{remark}

Let us denote by $J$ the functor that sends $V \in \calM(\frakg , K)$ to the submodule of $\frakn$-finite vectors in the $\theta(\frakn)$-completion of $V$ ($J$ is a version of the Casselman-Jacquet functor; See, for example, \cite{ENV}). Conjecture \ref{conj_dual} is equivalent to the following one:

\begin{conjecture}\label{conj_dual2}
The functors $\calC$ and $J$ are canonically isomorphic, perhaps up to tensoring by some power of $\wedge^{top} (\frakk / \frakm)$.
\end{conjecture}

Let us see why the conjectures \ref{conj_dual} and \ref{conj_dual2} are equivalent, \textbf{omitting everywhere powers of $\wedge^{top} (\frakk / \frakm)$}. Conjecture \ref{conj_dual} says that $\calC (V) \is (\calC^{-}(V^{\vee}))^{\vee}$ canonically, while conjecture \ref{conj_dual2} says that $\calC (V) \is J (V)$ canonically. Thus, in order to see that the conjectures are equivalent, we want to check that $(\calC^{-}(V^{\vee}))^{\vee} \is J (V)$ canonically. From proposition \ref{prop_relJC}, $\calC^{-}(V) \is J^{-*} (V^{\vee})$ (here, $J^{-*}$ is the functor $J^*$ w.r.t. the choice of Borel $\theta(B)$, rather than $B$). Thus, $(\calC^{-}(V^{\vee}))^{\vee} \is (J^{-*}(V))^{\vee}$. Now, recalling that $J^{-*}(V)$ is the space of $\theta(\frakn)$-torsion vectors in $V^*$, we see that we can identify the abstract dual of $J^{-*}(V)$ with the $\theta(\frakn)$-completion of $V$. Thus, we can identify $(J^{-*}(V))^{\vee}$ with the space of $\frakn$-torsion vectors in the $\theta(\frakn)$-completion of $V$, i.e. with $J(V)$.

\begin{remark}
	From the previous paragraph, we see that without tensoring by a power of $\wedge^{top} (\frakk / \frakm)$, the two conjectures are not exactly compatible; Something like the square root of $\wedge^{top} (\frakk / \frakm)$ will appear in some of them.
\end{remark}

\section{Regarding $\bbB^{der}$}

\begin{conjecture}\label{cnj_gbfa}
	The functor $$ \bbB^{der} : \calD (D_{\lambda} , MN) \to \calD(D_{\lambda} , K) $$ admits a $t$-exact right adjoint $$ \calD (D_{\lambda} , MN) \leftarrow \calD(D_{\lambda} , K) : \bbC^{der}.$$
\end{conjecture}

The following conjecture would imply conjecture \ref{conj_dual2}, and thus conjecture \ref{conj_dual}.

\begin{conjecture}\label{cnj_gjf}
	The functor $$ \calM (D_{\lambda} , MN) \leftarrow \calM(D_{\lambda} , K) : \bbC$$ is isomorphic to the geometric Jacquet functor of \cite{ENV} (when their functor is lifted to the correct equivariant categories) - again perhaps up to tensoring by some power of $\wedge^{top} (\frakk / \frakm)$.
\end{conjecture}

\section{Regarding the main statement}

\begin{conjecture}\label{cnj_geomlem}
	The $w$-th subquotient of the canonical filtration of the functor $$\bbC\circ \calP r^{\lambda} : \calM^{\lambda} (\frakt , M) \to \calM (D_{\lambda} , MN)$$ is isomorphic to $(i_w)_* \circ [\cdot]^{\lambda}_w$ (perhaps canonically given a choice of a lifting of $w \in N_G(T) / T)$ to an element of $K$). Thus, in turn, the $w$-th subquotient of the canonical filtration of the functor $$\calC \circ P r^{\lambda} :  \calM^{\lambda} (\frakt , M) \to \calM^{[\lambda]} (\frakg , MN)$$ is isomorphic to $V^{\lambda}_w$.
\end{conjecture}

A slightly neater and more general formulation would be as follows. Consider the extended universal enveloping algebra $\widetilde{U} = \widetilde{U}(\frakg) := U(\frakg) \otimes_{\calZ(\frakg)} S(\frakh)$ (see \cite{BG} for a more detailed discussion of it). It is straightforward to define $$ Pr : \calM (\frakt , M) \to \calM ( \widetilde{U} , K),$$ $$ V_w : \calM(\frakt , M) \to \calM(\widetilde{U} , MN)$$ and $$ \calC:  \calM ( \widetilde{U} , K) \to  \calM ( \widetilde{U} , MN).$$

\begin{conjecture}\label{cnj_geomlem2}
	The functor $$ \calC \circ Pr : \calM(\frakt , M) \to \calM(\widetilde{U} , MN)$$ admits a canonical $W$-filtration, whose $w$-th subquotient is isomorphic to $V_w$.
\end{conjecture}


\appendix


\pagebreak
\chapter{$D$-algebras and some functors}\label{app_dalg}

Our reference for $D$-algebras is \cite{BB}. Everything in this appendix is standard, except the functors $p_{\flat}$ and $p^{\flat}$.

\section{$D$-algebras}

Given a smooth variety $Y$, we have the notion of a $D$-algebra on $Y$. An important particular case of a $D$-algebra is that of a TDO. An important particular case of a TDO is the sheaf of differential operators $D_Y$.

Given a smooth morphism $p : Z \to Y$ and a $D$-algebra $\calA$ on $Y$, we have an induced $D$-algebra $p^{\cdot} \calA$ on $Z$. If $\calA$ is a TDO, then so is $p^{\cdot} \calA$. Also, $p^{\cdot} D_Y \is D_Z$.

If an affine algebraic group $K$ acts on $Y$, we have the notion of a $K$-equivariant $D$-algebra on $Y$. One possible way to define a $K$-equivariant $D$-algebra on $Y$ is as a $D$-algebra $\calA$, equipped with an isomorphism $prj^{\cdot} \calA \is act^{\cdot} \calA$ satisfying a suitable cocycle condition, where $prj, act : K \times Y \to Y$ are the projection and action maps.

Given a variety $Z$ on which $K$ acts, a smooth $K$-equivariant morphism $p: Z \to Y$, and a $K$-equivariant $D$-algebra $\calA$ on $Y$, the $D$-algebra $p^{\cdot} \calA$ has a natural $K$-equivariant structure.

Suppose that $p : Z \to Y$ is a principal $K$-fibration. By this we mean that the $K$-action on $Y$ is trivial and that etale-locally $p : Z \to Y$ looks like the projection $K \times Y \to Y$. Then, via $p^{\cdot}$, the category of $D$-algebras on $Y$ is equivalent to the category of $K$-equivariant $D$-algebras on $Z$, and we denote the inverse equivalence by $p_+$.

Finally, let us remind that if $\calL$ is a $K$-equivariant line bundle on $Y$, and $\calA$ a $K$-equivariant $D$-algebra on $Y$, then we have the $K$-equivariant $D$-algebra $\calA_{\calL}$ on $Y$ - the twist of $\calA$ by $\calL$ - which is defined as $End_{\calA} (\calL\otimes_{\calO} \calA)$ (where the $End$ is w.r.t. the $\calA$-action on the right).

\section{Modules over $D$-algebras}

Let $Y$ be a smooth variety and $\calA$ a $D$-algebra on $Y$. We have the notion of an $\calA$-module, and we denote by $Mod(\calA)$ the abelian category of $\calA$-modules.

If an affine algebraic group $K$ acts on $Y$, we have the notion of a (strongly) $K$-equivariant $\calA$-module, and we denote by $Mod(\calA, K)$ the abelian category of $K$-equivariant $\calA$-modules.

If $\calL$ is a $K$-equivariant line bundle on $Y$, and $\calA$ a $K$-equivariant $D$-algebra on $Y$, then we have an equivalence $$ \calL\otimes_{\calO} \cdot : Mod(\calA , K) \rightleftarrows Mod(\calA_{\calL} , K) : \calL^{-1} \otimes_{\calO} \cdot.$$

\section{The functors $p^{\circ} , p_+$}

Let $Y$ be a smooth variety and $\calA$ a $D$-algebra on $Y$. Let $p : Z \to Y$ be a smooth morphism. We have the functor $$p^{\circ} : Mod(\calA) \to Mod(p^{\cdot}\calA).$$ On the level of $\calO$-modules, this functor is just the usual pullback of quasi-coherent sheaves.

Suppose that an affine algebraic group $K$ acts on $Y$ and $Z$, and that $p$ is $K$-equivariant. Then the functor $p^{\circ}$ naturally lifts to a functor $$p^{\circ} : Mod( \calA , K) \to Mod(p^{\cdot} \calA , K).$$

Suppose that $p : Z \to Y$ is a principal $K$-fibration. Then the functor $$ p^{\circ} : Mod(\calA) \to Mod(p^{\cdot}\calA , K)$$ is an equivalence; Its inverse equivalence $$Mod(\calA) \leftarrow Mod(p^{\cdot}\calA , K)  : p_+$$ sends a sheaf $\calF$ to the sheaf of $K$-invariants of its pushforward to $Y$. If $p$ admits a section $s: Z \to Y$ (i.e. the principal $K$-fibration is trivialized), then $p_+$ is identified with $s^{\circ}$. If an additional affine algebraic group $M$ acts on $Z$ and $Y$, commuting with the action of $K$, then the equivalences above extend to equivalences $$ p^{\circ} : Mod(\calA , M) \rightleftarrows Mod(p^{\cdot} \calA , K \times M): p_+.$$

\section{The functors $p_*, p_{\flat}, p^{\flat}$}

Let $Y$ be a smooth variety and $\calA$ a $D$-algebra on $Y$. Let $U$ be a smooth affine variety. Denote by $p : U \times Y \to Y$ the projection. We have $p^{\cdot}\calA \is \calO_U \boxtimes \calA$. Denote by $D(U), O(U), T(U), \Omega^{\text{top}}(U)$ the global sections of the corresponding sheaves on $U$ (of differential operators, regular functions, vector fields, top differential forms).

Since $U$ is affine we can (and will) identify $p^{\cdot}\calA$-modules with $\calA$-modules with $D(U)$-action (in other words, $D(U)$-module objects in the category $Mod(\calA)$). Also, recall the equivalence $$(\cdot)^R : Mod(D_U \boxtimes \calA) \rightleftarrows Mod(D_U^{\text{op}} \boxtimes \calA) : \presuper{L}(\cdot)$$ which translates between left and right $D(U)$-modules; $(\cdot)^R$ is given by tensoring with $\Omega^{\text{top}}(U)$ over $O(U)$, while $\presuper{L}(\cdot)$ is given by tensoring with the inverse module.

\subsection{The non-equivariant case}

We consider four functors, given by the usual formalism of $Hom$ and $\otimes$ (and translation via $(\cdot)^R$ and $\presuper{L}{(\cdot)}$ when needed). The fourth one, $p^{\flat}$, we will not construct or use in this thesis, unless $Y = pt$.

\begin{enumerate}

\item The functor $$ Mod(\calA) \leftarrow Mod(p^{\cdot} \calA) : p_{\flat}$$ is defined by $$\calF \mapsto Hom_{D(U)} (O(U) , \calF).$$ In other words, the functor $p_{\flat}$ assigns to a $D(U)$-module the submodule of sections killed by $T(U)$.

\item The functor $$ p_{*} : Mod(p^{\cdot} \calA) \to Mod(\calA)$$ is defined by $$\calF \mapsto \calF^{R} \otimes_{D(U)} O(U).$$ In other words, the functor $p_{*}$ first converts a left $D(U)$-module to a right $D(U)$-module, and then takes its coinvariants under $T(U)$.

\item We have already encountered the functor $$p^{\circ} : Mod(\calA) \to Mod(p^{\cdot} \calA);$$ It is given by $$\calF \mapsto \calO (U) \otimes \calF.$$

\item In case $Y=pt$, the functor $$ Mod(p^{\cdot} \calA) \leftarrow Mod(\calA) : p^{\flat}$$ is defined by $$\calF \mapsto \presuper{L}{Hom (O(U) , \calF)}.$$

\end{enumerate}

We have the adjunctions $$p^{\circ} : Mod(\calA) \rightleftarrows Mod(p^{\cdot} \calA) : p_{\flat}$$ and $$p_* : Mod(p^{\cdot} \calA) \rightleftarrows Mod(\calA)  : p^{\flat}.$$

\subsection{The equivariant case}

Suppose that an affine algebraic group $K$ acts on $Y$ and on $U$ and consider the diagonal $K$-action on $U \times Y$. Then $p$ is $K$-equivariant. All four functors and adjunctions from the previous subsection carry over to this equivariant case, with one change: In the definition of the functor $p^{\flat}$, we should take $K$-finite vectors: $p^{\flat} (\calF) := \presuper{L}{\left(Hom (O(U) , \calF)^{K-\text{fin.}}\right)}$.

If the $K$-action on $U$ is homogeneous, we have a surjection $O(U) \otimes \frakk \to T(U)$, and thus the functor $p_{\flat}$ can be described as taking sections killed by $\frakk$, and the functor $p_*$ can be described as taking coinvariants under $\frakk$ (after the translation to a right $D(U)$-module).


\pagebreak
\chapter{Filtrations and Stratifications}\label{app_str}

We fix a field $k$, and a finite partially ordered set $W$. We use the topology on $W$, where a subset $I \subset W$ is closed if $w \in I$ and $v \leq w$ imply $v \in I$. We write, for $w \in W$, $\overline{w} = \{ v \ | \ v \leq w \}$ (the closure in the topology). \quash{The set $Cl(W)$ of closed subsets of $W$ is itself a finite partially ordered set (w.r.t. containment).}

Throughout, $\calP$ and $\calQ$ will denote $k$-linear abelian categories, and all functors will be assumed $k$-linear.

\section{Filtrations on objects}

\begin{definition}
A \textbf{$W$-filtration} on an object $A \in \calP$ is the data of a subobject $A_I \subset A$ for every $I \in Cl(W)$, such that:

\begin{enumerate}[\quad	(1)]
	\item For $I \subset J$, we have $A_I \subset A_J$.
	\item $A_{\emptyset} = 0$ and $A_{W} = A$.
	\item For $I,J$, the morphism $A_I / A_{I \cap J} \oplus A_J / A_{I \cap J} \to A_{I \cup J} / A_{I \cap J}$ (the direct sum of the two natural embeddings) is an isomorphism.
\end{enumerate}

\end{definition}

\begin{remark}
	Condition (3) is equivalent to the condition: For $I,J$, we have $A_{I \cup J} = A_I + A_J$ and $A_{I \cap J} = A_I \cap A_J$. Thus, one can reformulate the definition of a $W$-filtration as a poset morphism $(Cl(W) , \subset ) \to (Sub(A) , \subset)$ which preserves finite supremums and infimums (here, $Sub(A)$ is the class of subobjects of $A$).
	
	Another equivalent formulation of condition (3) is that for $I,J$, the map $A_I / A_{I \cap J} \to A_{I \cup J} / A_{J}$ is an isomorphism.
\end{remark}

Suppose given a $W$-filtration on an object $A$. For a locally closed subset $K \subset W$, i.e a subset of the form $K = I - J$ where $J \subset I$ are closed, we define $A^K := A_{\overline{K}} / A_{\overline{K} - K}$. This also is canonically isomorphic to $A_I / A_J$ in the notation above, by condition (3) in the definition of a $W$-filtration. In particular, we define the \textbf{$w$-th subquotient} $A^w := A^{\{ w\} }$.

\section{Stratifications on categories}

\begin{definition}
A \textbf{$W$-stratification on $\calP$} is the data of a Serre subcategory $\calP_I \subset \calP$ for every $I \in Cl(W)$, such that:
\begin{enumerate}[\quad	(1)]
\item For $I \subset J$, we have $\calP_I \subset \calP_J$, and the projection $\calP_J \to \calP_J / \calP_I$ admits a left and a right adjoint.
\item $\calP_{\emptyset} = 0$ and $\calP_{W} = \calP$.
\item For $I,J$, the functor $\calP_I / \calP_{I \cap J} \oplus \calP_J / \calP_{I \cap J} \to \calP_{I \cup J} / \calP_{I \cap J}$ (the direct sum of the two natural embeddings) is an equivalence.
\end{enumerate}
\end{definition}

\begin{remark}
The above definition is taken from \cite[Section 2.6.4]{BB}.
\end{remark}

Suppose given a $W$-stratification on $\calP$. For a locally closed subset $K \subset W$, i.e a subset of the form $K = I - J$ where $J \subset I$ are closed, we define $\calP^K := \calP_{\overline{K}} / \calP_{\overline{K} - K}$. This also is canonically equivalent to $\calP_I / \calP_J$ in the notation above, by condition (3) in the definition of a $W$-stratification. In particular, we define $\calP^w := \calP^{\{ w\} }$. Denote by $$i_w^{\bullet} : \calP_{\overline{w}} \to \calP^w$$ the projection, by $(i_w)_!$ its left adjoint and by $(i_w)_*$ its right adjoint. It is a standard fact that the unit and counit morphisms, $id \to i_w^{\bullet} \circ (i_w)_!$ and $i_w^{\bullet} \circ (i_w)_* \to id$, are isomorphisms.

\begin{definition}

$\calP$ is called \textbf{smallish} if:

\begin{enumerate}[\quad	(1)]
\item Every object in $\calP$ is of finite length.
\item There are finitely many isomorphism classes of irreducible objects in $\calP$.
\item Every irreducible object $E \in \calP$ satisfies $End(E) \is k$.
\end{enumerate}

\end{definition}

\begin{definition}

A $W$-stratification on $\calP$ is called \textbf{simple} if:
	
\begin{enumerate}[\quad	(1)]
\item For every $w \in W$, the category $\calP^w$ is smallish and semisimple.
\item For every $v,w \in W$ and $E \in \calP^v , F \in \calP^w$, we have $Ext^2_{\calP} ((i_v)_! E , (i_w)_* F) = 0$.
\end{enumerate}

In addition, the $W$-stratification is called \textbf{affine}, if the functors $(i_w)_!$ and $(i_w)_*$ are exact, for every $w \in W$.

\end{definition}

\begin{theorem}\label{thm_strat} Suppose given a simple $W$-stratification on $\calP$. Then:
\begin{enumerate}
\item $\calP$ is smallish.
\item The irreducible objects in $\calP$ are exactly the objects of the form $(i_w)_{!*} E := Im((i_w)_! E \to (i_w)_* E)$, where $w \in W$ and $E \in \calP^w$ is irreducible.
\item $\calP$ has enough projective objects, and every projective object in $\calP$ admits a filtration whose subquotients are in the image of $(i_w)_!$ for various $w \in W$.
\item $\calP$ has enough injective objects, and every injective object in $\calP$ admits a filtration whose subquotients are in the image of $(i_w)_*$ for various $w \in W$.
\item $\calP$ has finite cohomological dimension.
\item Every right exact functor $F : \calP \to (FinVect_k)^{op}$ is representable (here, $FinVect_k$ is the smallish category of finite-dimensional vector spaces over $k$). In particular, every right exact functor $F: \calP \to \calQ$ admits a right adjoint.
\end{enumerate}
\end{theorem}

\begin{proof}
	For parts 3,4,5, we refer to \cite[Section 3.2]{BGS}. In order to fit in their setup, we need to verify parts 1,2. This verification is quite a standard nice exercise in playing with the various adjunctions - one works recursively, each time focusing on a subcategory $\calP_{I}$ and its quotient $\calP^{w}$, where $I \in Cl(W)$ and $w \in I$ is maximal. For part 6, see for example \cite[Proposition 2.4]{MV}.
\end{proof}

\subsection{Standard objects in stratified categories}

Suppose that we are given an affine simple $W$-stratification on $\calP$.

\begin{definition}
An object $A \in \calP$ is called:
\begin{itemize}
\item \textbf{standard}, if $A$ is isomorphic to $(i_w)_! E$ for some $w \in W$ and irreducible $E \in \calP^w$.
\item \textbf{costandard}, if $A$ is isomorphic to $(i_w)_* E$ for some $w \in W$ and irreducible $E \in \calP^w$.
\item \textbf{standard-filtered}, if $A$ admits a filtration by standard objects.
\item \textbf{costandard-filtered}, if $A$ admits a filtration by costandard objects.
\end{itemize}
\end{definition}

\begin{remark}
In theorem \ref{thm_strat} it is said that projective objects are standard-filtered, and injective objects are costandard-filtered.
\end{remark}

Let us notice that all the statements and definitions below have dual ones, which we do not state explicitly.

\begin{definition}\label{def_BigEnough}
	A class of objects $\calO \subset \calP$ is called (left-)\textbf{big enough}, if the following holds:
	\begin{enumerate}[(1)]
		\item $\calO$ is closed under finite direct sums.
		\item If in a short exact sequence $$0 \to A \to B \to C \to 0$$ we have $B,C \in \calO$, then $A \in \calO$.
		\item Every object in $\calP$ can be presented as a quotient of an object in $\calO$.
	\end{enumerate}

\end{definition}

\begin{remark}
	Classes of objects which are big enough are important for constructing derived functors. Namely, given a right exact functor $F: \calP \to \calQ$, a class of objects $\calO \subset \calP$ is said to be \textbf{adapted} for $F$, if it is big enough, and $F$ transforms a short exact sequence of objects in $\calO$ into a short exact sequence. In that case, the left derived functor $LF : D^b (\calP) \to D^b (\calQ)$ can be constructed by the rule of applying $F$ term-wise to complexes of objects in $\calO$. In particular, $LF$ sends objects from $\calO$ to acyclic objects.
\end{remark}

\begin{lemma}\label{lem_BigEnough}
	The class of standard-filtered objects in $\calP$ is big enough, and is adapted to the functors $Hom(\cdot , C)$, where $C$ is costandard-filtered.
\end{lemma}

\begin{lemma}\label{lem_charstflt}
For $A \in \calP$, the following are equivalent:
\begin{enumerate}
\item $A$ is standard-filtered.
\item For any costandard-filtered object $C \in \calP$, the $Hom$-space $RHom(A ,C)$ is acyclic.
\end{enumerate}

\end{lemma}

\begin{lemma}
A standard-filtered object $A \in \calP$ has a canonical $W$-filtration, characterized as the unique $W$-filtration on $A$ for which the $w$-th subquotient is in the image of $(i_w)_!$, for every $w \in W$.
\end{lemma}

\begin{remark}\label{rem_funcfilt}
	In corollary \ref{cor_main} in the main text, the following remark will be useful. Let $F: \calP \to \calQ$ be a functor, with the property that it transforms a short exact sequence whose first term is costandard-filtered, into a short exact sequence. Then for a costandard-filtered $A \in \calP$, denoting by $I \mapsto A_I$ the canonical $W$-filtration mentioned in the previous lemma, we have that $I \mapsto F(A_I)$ is a $W$-filtration on $F(A)$, whose $w$-th subquotient is canonically isomorphic to $F$ applied to the $w$-th subquotient of $A$.
\end{remark}

\subsection{The derived setting}

\begin{claim} Let $\calD$ be a bounded $t$-category equipped with a $t$-exact functor $\iota: D^b (\calP) \to \calD$ inducing an equivalence on the hearts (which we will just identify notationally). Suppose that we are given a $W$-stratification on $\calP$.

\begin{enumerate}
\item For any two $A,B \in \calP$, the map $$Ext_{\calP}^2 (A,B) = Hom_{D^b(\calP)} (A,B[2]) \to Hom_{\calD} (A,B[2])$$ is injective, and hence condition (2) in the definition of a simple stratification is implied by the similar condition $$Hom_{\calD} ((i_v)_! E , (i_w)_* F [2]) = 0.$$
\item If the $W$-stratification on $\calP$ is simple and $Hom_{\calD}(S,C[j]) = 0$ for every standard $S \in \calP$, costandard $C \in \calP$ and $j > 0$, then $\iota$ is an equivalence.
\end{enumerate}

\end{claim}

\begin{proof}
For part (1) see \cite[Lemma 3.2.4]{BGS}, and for part (2) see \cite[Corollary 3.3.2]{BGS}.
\end{proof}


\pagebreak
\chapter{Sketch of proof of part $(b)$ of theorem \ref{thm_main}}\label{app_detail}

In this appendix we sketch the proof of part $(b)$ of theorem \ref{thm_main}.

For the following preparatory claim we just sketch the reasoning, without providing detail.

\begin{claim}\label{clm_pseqgrth}
	Let $F \in \calM^{\lambda} (\frakt , M)$. The objects $$ \bbB ((i_w)_! [F]^{\lambda}_w) \in \calM(D_{\lambda} , K),$$ for different $w \in W$, are all equal in the Grothendieck group.
\end{claim}

\begin{proof}
	We can first, by tensoring with a suitable $G$-equivariant line bundle, reduce to the case when $\lambda$ is regular antidominant.
	
	Next, the global sections of $$ \bbB ((i_w)_! [F]^{\lambda}_w)$$ are the same as those of  $$ I_{*,w^{-1}} \bbB ((i_w)_! [F]^{\lambda}_w) \is \bbB ((i_1)_! [F]^{w^{-1} \lambda}_1) = \calP r^{w^{-1} \lambda} (F).$$ Thus, in order to verify the claim, we should verify that the Harish-Chandra modules $$ Pr^{w\lambda}(F) \in \calM^{[\lambda]}(\frakg , K),$$ for different $w \in W$, are all equal in the Grothendieck group. This is a well-known property of the principal series representations, that principal series representations with the same central and infinitesimal characters are equal in the Grothendieck group.
\end{proof}

\begin{proof}[Proof (of part $(b)$ of theorem \ref{thm_main})]

To see what is the $w$-th subquotient in the canonical $W$-filtration of $\bbC(\calP r^{\lambda} (E))$, we can assume without loss of generality that $E$ is irreducible. We want to see that, for an irreducible $F \in \calM^{\lambda}(\frakt , M)$, the space $$Hom ((i_w)_! [F]^{\lambda}_w , \bbC(\calP r^{\lambda} (E)))$$ is one-dimensional if $F$ is isomorphic to $E$, and is zero otherwise. By adjunction, this is the same as $$Hom (\bbB((i_w)_! [F]^{\lambda}_w) , \calP r^{\lambda} (E)).$$

First, let us notice that we can replace the dimension of $$Hom (\bbB((i_w)_! [F]^{\lambda}_w) , \calP r^{\lambda} (E))$$ with the Euler characteristic of $$RHom_{\calD(D_{\lambda},K)}(\bbB((i_w)_! [F]^{\lambda}_w) , \calP r^{\lambda} (E)).$$ Indeed, the later is acyclic because, recalling that $\calP r^{\lambda}(E) = j_* [E]^{\lambda}$, we have $$RHom_{\calD(D_{\lambda},K)}(\bbB((i_w)_! [F]^{\lambda}_w) , \calP r^{\lambda} (E)) \is RHom_{\calD(j^{\cdot} D_{\lambda},K)}(j^* \bbB((i_w)_! [F]^{\lambda}_w) , [E]^{\lambda}),$$ and the later is clearly acyclic.

Second, we can replace the Euler characteristic of $$RHom_{\calD(D_{\lambda},K)}(\bbB((i_w)_! [F]^{\lambda}_w) , \calP r^{\lambda} (E))$$ with the Euler characteristic of $$RHom_{\calD(D_{\lambda},K)}(\bbB((i_1)_! [F]^{\lambda}_1) , \calP r^{\lambda} (E)).$$ Indeed, this follow from claim \ref{clm_pseqgrth}.

Now, the later is just $$RHom_{\calD(D_{\lambda},K)}( j_* [F]^{\lambda} , j_* [E]^{\lambda})$$ which, by adjunction, can be written as $$RHom_{\calD(j^{\cdot}D_{\lambda},K)}( [F]^{\lambda} ,  [E]^{\lambda}).$$ This clearly has Euler characteristic $1$ if $F$ is isomorphic to $E$, and $0$ otherwise.

\end{proof}


\end{document}